\documentclass{article}
\usepackage[english]{babel}
\usepackage{a4wide}
\usepackage[T1]{fontenc}
\usepackage{amsmath}
\usepackage{amssymb}
\usepackage{amsthm}
\usepackage{graphicx}
\usepackage{mathrsfs}

\def\N{\mathbb{N}}

\def\R{\mathbb{R}}
\def\D{\mathsf{D}}
\def\L{\textnormal{L}}

\def\H{\textnormal{H}}
\def\W{\textnormal{W}}
\def\pa{\partial}

\def\epsilon{\varepsilon}
\def\ov{\overline}
\def\a{\alpha}
\def\b{\beta}
\def\c{\gamma}
\def\d{\delta}
\def\O{\Omega}

\newcommand{\dd}{\mathrm{d}}

\newtheorem{Theo}{Theorem}[section]
\newtheorem{cor}[Theo]{Corollary}
\newtheorem{lem}[Theo]{Lemma}

\newtheorem{definition}[Theo]{Definition}
\newtheorem{remark}[Theo]{Remark}
\def\footnote{\@ifnextchar[{\@xfootnote}{\stepcounter {\@mpfn}\xdef\@thefnmark{\thempfn}\@footnotemark\@footnotetext}}
\newcommand{\ud}{\mathrm{d}}

\begin{document}

\begin{center}
\Large{On triangular reaction cross-diffusion systems with possible self-diffusion}
 \end{center}
\bigskip

\medskip
\centerline{\scshape A. Trescases}
\medskip
{\footnotesize
  \centerline{CMLA, ENS Cachan, CNRS}
  \centerline{61 Av. du Pdt. Wilson, F-94230 Cachan, France}
\centerline{E-mail: trescase@cmla.ens-cachan.fr}}
\bigskip

\begin{abstract} We present new results of existence of global solutions for a class of reaction cross-diffusion systems of two equations presenting a cross-diffusion term in the first equation, and possibly presenting a self-diffusion term in any (or both) of the two equations. This class of systems arises in Population Dynamics, and notably includes the triangular SKT system. In particular, we recover and extend existing results for the triangular SKT system. Our proof relies on entropy and duality methods.
\end{abstract}

%
%


\bigskip

\section{Introduction}
The purpose of this paper is to investigate existence and some properties of the solutions of the system
\begin{align}
\label{sku1} \pa_t u - \Delta_x [(d_u + d_\a u^\a + d_\b v^\b ) \, u] = u\, (r_u - r_{a}\, u^a - r_{b}\,v^b)& \qquad \text{in } \R_+\times\Omega,\\
\label{sku2} \pa_t v - \, \Delta_x [(d_v + d_\c v^\c)\, v]  = v\, (r_v - r_{c}\,v^c - r_d\, u^d)& \qquad \text{in } \R_+\times\Omega,\\
\label{sku3} \nabla_x u \cdot n = \nabla_x v\cdot n = 0& \qquad \text{on } \R_+\times\pa\Omega,\\
\label{sku4} u(0,\cdot)=u_{in}, \qquad v(0,\cdot)=v_{in}& \qquad \text{in } \Omega,
 \end{align}
where $u=u(t,x)\ge 0$, $v=v(t,x)\ge0$ are the unknowns, the variables $(t,x)$ browse $\R_+\times\Omega$ with $\Omega$ a bounded domain of $\R^m$ ($m\ge1$), $n=n(x)$ stands for the outward normal at point $x$ of the boundary $\pa \O$, $u_{in}$ and $v_{in}$ are nonnegative initial data, and the remaining terms are nonnegative constant parameters satisfying
\begin{equation}\label{cdt:param}\begin{split}
\D:=\{d_u,d_v,d_\a,d_\b,d_\c,r_u,r_v,r_a,r_b,r_c,r_d,a,b,c,d,\a,\b,\c\}\in (\R_+^\ast)^{15}&\times\R_+\times\R_+^\ast\times\R_+,\\
\text{($\a>0$, $d<2+\a$,  $a<1+\a$) or ($\a=0$, $d\le 2$,  $a\le 1$).}&
\end{split}
\end{equation}

The origin of this system is to be found in a well-studied system arising in Population Dynamics, known in the literature as the SKT system. 
The SKT system was introduced in \cite{SKT} to model spatial segregation in two competing species of (let us say) animals (see also \cite{Okubo}). It has since then attracted the interest of many mathematicians, leading to a rich literature on the question of the existence of solutions (see for example \cite{ChenJungel06} and references therein) and on the analysis of equilibria and stability (patterns are shown to appear; see for example \cite{IMN}). Writing $u$ and $v$ the respective densities of the two different species, it takes the following form
\begin{equation}\label{sktoriginal}\begin{split}
\pa_t u - \Delta_x [(d_u + d_\a u + d_\b v ) \, u] = u\, (r_u - r_{a}\, u - r_{b}\,v) \qquad \text{in } \R_+\times\Omega,\\
\pa_t v - \, \Delta_x [(d_v + d_\c v + d_\d u)\, v]  = v\, (r_v - r_{c}\,v - r_d\, u) \qquad \text{in } \R_+\times\Omega,\\
\nabla_x u \cdot n = \nabla_x v\cdot n = 0 \qquad \text{on } \R_+\times\pa\Omega.
\end{split}\end{equation}
When $d_u=d_v=d_\a=d_\b=d_\c=d_\d=0$, the system \eqref{sktoriginal} reduces to the standard Lotka-Volterra competition ODS. The terms $r_u$, $r_v$ are the intrinsic growth of the species, while $r_b$ and $r_d$ measure the demographic effect of the interspecific competition, and $r_a$, $r_c$ indicate the demographic effect of the intraspecific competition.
When non-zero, the terms $\Delta_x [(d_u + d_\a u + d_\b v ) \, u]$ and $\Delta_x [(d_v + d_\c v + d_\d u)\, v]$ model the spatial movements of the individuals in the domain $\O$. The positive constants $d_u$, $d_v$ are standard diffusion rates, which indicate the frequency of the random walk of the individuals inside each species. The nonnegative constants $d_\b$ and $d_\d$ are usually referred to as "cross-diffusion" coefficients, and ecologically measure the repulsive effect, on the individuals of one species, of the presence of the individuals of the other species (as a result of the interspecific competitive pressure). The nonnegative constants $d_\a$ and $d_\c$, referred to as "self-diffusion" coefficients, measure the repulsive effect on the individuals of the presence of the individuals of the same species (as a result of the intraspecific competitive pressure).

The system \eqref{sktoriginal} is often called "triangular" when $d_\d=0$. From the point of view of modeling, this means that only the individuals of the first species tend to avoid the individuals of the other species. 
It therefore takes the form
\begin{equation}\label{sktoriginaltriang}\begin{split}
\pa_t u - \Delta_x [(d_u + d_\a u + d_\b v ) \, u] = u\, (r_u - r_{a}\, u - r_{b}\,v) \qquad \text{in } \R_+\times\Omega,\\
\pa_t v - \, \Delta_x [(d_v + d_\c v)\, v]  = v\, (r_v - r_{c}\,v - r_d\, u) \qquad \text{in } \R_+\times\Omega,\\
\nabla_x u \cdot n = \nabla_x v\cdot n = 0 \qquad \text{on } \R_+\times\pa\Omega.
\end{split}\end{equation}
In this case, the second equation is coupled to the first one only through zeroth-order terms (reaction), while in the full system \eqref{sktoriginal} with $d_\d>0$ both equations are coupled to the other one through both zeroth-order (reaction) and second-order terms (cross-diffusion). The full system \eqref{sktoriginal} has a completely different structure from the triangular case (in particular it is possible to exhibit an entropy structure when $d_\d>0$, see \cite{ChenJungel06} and \cite{DLMT}, but this entropy structure degenerates when $d_\d=0$).

In our study, we consider the class of systems \eqref{sku1}--\eqref{sku3}. That is, we focus on a triangular type of cross-diffusion ($d_\d=0$) as in the system \eqref{sktoriginaltriang}, but in contrast to the system \eqref{sktoriginaltriang} where the diffusion rates $(d_u + d_\a u + d_\b v,\, d_v + d_\c v + d_\d u)$ and the growth rates $(r_u - r_{a}\, u - r_{b}\,v,\,r_v - r_{c}\,v - r_d\, u)$ are required to be linear functions of $u$ and $v$, in \eqref{sku1}--\eqref{sku3} we allow these functions to be more general power laws (with suitably chosen powers, \eqref{cdt:param}). Note that the class of systems we consider includes \eqref{sktoriginaltriang} (when $\a=\b=\c=a=b=c=d=1$).\\

We now present our main mathematical result for this class of systems.

\subsection{Main Theorem}

We clarify the notion of weak solution we will use in the
\begin{definition}\label{def:weaksol}
Let $\Omega$ be a smooth bounded domain of $\mathbb{R}^m$ ($m\in \N^\ast$) and let $\D\in (\R_+^\ast)^{15}\times\R_+\times\R_+^\ast\times\R_+$.
Let $u_{in}:=u_{in}(x)\ge 0$ and $v_{in}:=v_{in}(x)\ge 0$ be two functions lying in $L^1(\Omega)$.
\par
 A couple of functions $(u,v)$ such that $u:=u(t,x)\ge0$ and $v:=v(t,x)\ge0$, and lying in $L_{\text{loc}}^{\max(1+a,d)}(\R_+\times\ov{\Omega})\times L_{\text{loc}}^{\infty}(\R_+\times\ov{\Omega})$ is a (global) {\bf weak solution} of \eqref{sku1}-\eqref{sku4} 
 if 
$$\nabla_x \left[(d_u+ d_\a u^\a+d_\b v^\b)\,u\right],\;\nabla_x \left[(d_v+ d_\c v^\c)\,v\right] \;\in\; L_{\text{loc}}^{1}(\R_+\times\ov{\Omega})$$
and, for all test functions $\psi_1$, $\psi_2 \in C^1_c(\R_+ \times {\ov{\Omega}})$, we have the identities
\begin{equation}\label{eq:weak_form}
 - \int_0^{\infty}\int_{\Omega} (\pa_t \psi_1)\, u - \int_{\Omega} \psi_1(0,\cdot)\, u_{in} 
+ \int_0^{\infty}\int_{\Omega} \nabla_x \psi_1 \cdot \nabla_x \left[(d_u+ d_\a u^\a+d_\b v^\b)\,u\right] =  \int_0^{\infty}\int_{\Omega} \psi_1\, u\, (r_u - r_{a}\, u^a - r_{b}\,v^b), 
\end{equation}
\begin{equation*}
 - \int_0^{\infty}\int_{\Omega} (\pa_t \psi_2)\, v - \int_{\Omega} \psi_2(0,\cdot)\, v_{in} 
+ \int_0^{\infty}\int_{\Omega} \nabla_x \psi_2 \cdot  \nabla_x \left[(d_v+ d_\c v^\c)\,v\right] = \int_0^{\infty}\int_{\Omega} \psi_2\,
 v\, (r_v - r_{c}\,v^c - r_d\, u^d). 
 \end{equation*}
Note that the assumptions on $u_{in}$, $v_{in}$, $u$, $v$, $\psi_1$, $\psi_2$ ensure that all integrals in the two identities above are finite.
\end{definition}

Our main result is contained in the

\begin{Theo}\label{theo_ex_csd}
Let $\Omega$ be a smooth bounded domain of $\mathbb{R}^m$ ($m\in \N^\ast$). 
Let the coefficients of system (\ref{sku1}) -- (\ref{sku2}) satisfy
condition \eqref{cdt:param}.
Consider initial data $u_{in}\ge 0$, $v_{in}\ge 0$ such that $u_{in} \in \L^{2}(\Omega)$
,  $v_{in}\in \L^\infty(\Omega)$.
\medskip

Then,\\
\textbf{i)} there exists $u=u(t,x)\ge0$, $v=v(t,x)\ge0$ such that $(u,v)\in L_{\text{loc}}^{2+\a}(\R_+\times\ov{\Omega})\times L_{\text{loc}}^{\infty}(\R_+\times\ov{\Omega})$ and $(u,v)$ is a (global) weak solution of system (\ref{sku1}) -- (\ref{sku4}) in the sense of Definition~\ref{def:weaksol}.
\par
Furthermore, this solution $(u,v)$ satisfies for all $T>0$
\begin{eqnarray}\label{th_es1}
\sup_{[0,T]\times\O} v \le \max\left\{ \sup_\Omega v_{in}, \left(\frac{r_v}{r_c}\right)^{1/c} \right\},\\
\label{th_es2}\int_0^T\int_{\O} u^{2+\a}\leq C(\Omega,T, u_{in},v_{in},\D),\\
\label{th_es3}\int_0^T \int_\O \left|\nabla_x [v^{p/2}]\right|^2
\le C(p,\O,T,u_{in},v_{in},\D) \qquad (\text{for all } 0<p<\infty),\\
\int_\O u (T) \, + \, \int_0^T \int_\O \left|\nabla_x [(1+u)^{\a/2}]\right|^2
\le C(\O,T,u_{in},v_{in},\D) \qquad (\text{if } \a>0),\\
\label{th_es5}\int_\O u (T) \, + \, \int_0^T \int_\O \left|\nabla_x [\log(1+u)]\right|^2
\le C(\O,T,u_{in},v_{in},\D) \qquad (\text{if } \a=0),
\end{eqnarray}
where the constant $C(\O,T,u_{in},v_{in},\D)$ only depends on the domain $\O$ (and the dimension $m$), the time $T$, the norms of the initial data $\|u_{in}\|_{\L^2(\Omega)}$ and $\|v_{in}\|_{\L^\infty(\O)}$ and the choice of parameters $\D$, and the constant $C(p,\O,T,u_{in},v_{in},\D)$ only depends on the same quantities and the parameter $p$.\\
If $u_{in}$ furthermore satisfies $u_{in}(x)>0$ a.e. on $\Omega$ and $\log u_{in}\in \L^1(\Omega)$, resp., if $v_{in}$ furthermore satisfies $v_{in}(x)>0$ a.e. on $\Omega$ and $\log v_{in}\in \L^1(\Omega)$, then for all $T>0$
\begin{eqnarray}
\label{th_es6u}
\int_\O |\log u| (T) \, + \, d_u \int_0^T \int_\O \left|\nabla_x [\log u]\right|^2
\le  \int_\O |\log u_{in}| \, + \,  C(\O,T,u_{in},v_{in},\D),\\
\label{th_es6v}\text{resp., }\quad \int_\O |\log v| (T) \, + \, d_v \int_0^T \int_\O \left|\nabla_x [\log v]\right|^2
\le \int_\O |\log v_{in}| \, + \,  C(\O,T,u_{in},v_{in},\D).
\end{eqnarray}
\textbf{ii)} If $\a=0$, we have the additional estimate, for some $\nu=\nu(\O,v_{in},\D)>0$ depending only on the domain $\O$ (and $m$), the norm $\|v_{in}\|_{\L^\infty(\O)}$ and the parameters $\D$, for all $T>0$, and for some $C_1(\Omega,T, u_{in},v_{in},\D)>0$ depending only on the domain $\O$ (and $m$), the time $T$, the norms $(\|u_{in}\|_{\L^2(\Omega)},\,\|v_{in}\|_{\L^\infty(\O)})$ and the parameters $\D$,
\begin{equation}\label{th_es7}
\int_0^T\int_{\O} u^{2+\nu}\leq C_1(\Omega,T, u_{in},v_{in},\D).
\end{equation}
\textbf{iii)} If $\c=0$, assuming furthermore that $v_{in}\in \W^{2,q}(\O)$ for some $q>1$ satisfying
\begin{equation*}
1<q\le (2+\a)/d \quad \text{if } \a>0, \qquad 1<q\le (2+\nu)/d \quad \text{if } \a=0,
\end{equation*}
(and assuming furthermore the compatibility condition $\nabla_x v_{in}\cdot n=0$ on $\pa \O$ if $q\ge 3$), we have the additional estimate, for all $T>0$,
\begin{equation}\label{th_es8}
\int_0^T \int_\O |\pa_t v|^{q} \,+\,\int_0^T \int_\O |\nabla_x^2 v|^{q} \,+\, \int_0^T \int_\O \left|\nabla_x v \right|^{2q}
\le C_2(q,\O,T,u_{in},v_{in},\D),
\end{equation}
where the constant $C_2(q,\O,T,u_{in},v_{in},\D)$ only depends on the parameter $q$, the domain $\O$ (and the dimension $m$), the time $T$, the norms of the initial data $(\|u_{in}\|_{\L^2(\Omega)},\,\|v_{in}\|_{\L^\infty\cap \W^{2,q}(\O)})$ and the parameters $\D$.
\end{Theo}

We list in the remark below some possible extensions of the results mentioned in Theorem \ref{theo_ex_csd}.
\begin{remark}
Thanks to the bound \eqref{th_es1} given by a maximum principle for equation \eqref{sku2}, the power laws "$v\mapsto v^\b$" and "$v\mapsto v^b$" in \eqref{sku1} can be replaced by any continuous function of $v$ on $\R_+$ which is smooth ($C^1(\R_+^\ast)$) and positive-valued on $\R_+^\ast$. The power law "$v\mapsto v^\c$" in \eqref{sku2} can be replaced by any continuous function on $\R_+$ which is non-decreasing, smooth, and positive-valued on $\R_+^\ast$. Following the same idea, but furthermore ensuring that the maximum principle remains valid for \eqref{sku2}, the power law "$v\mapsto v^c$" can be replaced by any continuous function on $\R_+$ which is smooth, positive-valued on $\R_+^\ast$ and which furthermore tends to $+\infty$ in $+\infty$. With these replacements, all results of Theorem \ref{theo_ex_csd} hold, with the bound \eqref{th_es1} adapted when necessary, and with the condition $\c=0$ in iii) replaced by the condition that the function replacing "$v\mapsto v^\c$" is constant.\\
The power laws "$u\mapsto u^a$" and "$u\mapsto u^d$" can be replaced by any continuous function of $u$ on $\R_+$, smooth and positive-valued on $\R_+^\ast$, and dominated by (or having the same behaviour as) "$u\mapsto u^a$", resp. "$u\mapsto u^d$" in $+\infty$. The power law "$u\mapsto u^\a$" can be replaced by any non-decreasing continuous function of $u$ on $\R_+$ which is smooth and positive-valued on $\R_+^\ast$ and have the same behaviour as "$u\mapsto u^\a$" in $+\infty$. With these replacements, all results of Theorem \ref{theo_ex_csd} hold.\\

Assume $\a=0$. In this case, an estimate of type \eqref{th_es2} ensures enough integrability (to get compactness) for the terms $r_a u^{1+a}$ in \eqref{sku1} and $r_d v u^d$ in \eqref{sku2} only when $a<1$ and $d<2$. The slightly better estimate \eqref{th_es7} is therefore crucial to estimate these terms in the two cases ($a=1$, $d\le2$) and ($a<1$, $d=2$). It even treats these terms when $a<1+\nu$ and $d<2+\nu$. As a consequence, we expect the results of Theorem \ref{theo_ex_csd} to hold when the condition ($\a=0$, $a\le 1$, $d\le2$) is replaced by the wider condition ($\a=0$, $a<1+\nu$, $d<2+\nu$). Note that $\nu$ can indeed be chosen independent of $a$ and $d$ (see Section \ref{subsec:alpha0}).\\

Assume $\a\ge 0$ and $1+a,\,d<2+\a$. As seen in Sections \ref{sec:apriori}--\ref{sec:existence}, in this case the proof of existence entirely relies on estimates of the type \eqref{th_es1}--\eqref{th_es5} (in particular we do not use estimate \eqref{th_es7}). As shown in the sequel, the proof of estimates of the type \eqref{th_es1}--\eqref{th_es5} only requires $u_{in}$ to be in $\L^1\cap\H^{-1}_m(\Omega)$, where the space $\H^{-1}_m(\Omega)$ is defined in Section \ref{subsec:notations}. Therefore in the case $\a\ge 0$ and $1+a$, $d<2+\a$, the results i) and iii) for any $1<q\le (2+\a)/d$ hold with the assumption $u_{in}\in \L^2(\Omega)$ replaced by the weaker assumption $u_{in}\in \L^1\cap\H^{-1}_m(\Omega)$ (and with the constants $C$ and $C_2$ depending on $\|u_{in}\|_{\L^1\cap\H^{-1}_m(\Omega)}$ instead of $\|u_{in}\|_{\L^2(\Omega)}$).\\

Assume $\a=0$. In ii), estimate \eqref{th_es7} is a consequence of a Lemma relying on duality techniques (namely, Lemma \ref{jeveuxunnom}), which is adapted from a similar result from \cite{CDF14}. This Lemma can be somewhat improved. Indeed, following Remark 2.3 in \cite{CDF14}, the constraint $u_{in}\in \L^2(\O)$ can be replaced by the weaker constraint $u_{in}\in \L^{2-\nu_1}(\O)$ where $\nu_1$ is a small positive constant. Therefore, all results of Theorem \ref{theo_ex_csd} hold when $u_{in}$ only belongs to the space $\L^{2-\nu_1}(\O)$ (where $\nu_1$ only depends on $\O$, $m$, $\|v_{in}\|_{\L^\infty(\O)}$ and $\D$), with the constants $C_1$ and $C_2$ depending on $\|u_{in}\|_{\L^{2-\nu_1}(\O)}$ instead of $\|u_{in}\|_{\L^{2-\nu_1}(\O)}$.\\

Assume $\c=0$.  Estimate \eqref{th_es8} being a direct consequence of the properties of the heat kernel (see Section \ref{subsec:gamma0}), we can replace the set $\W^{2,q}(\O)$ in iii) by the optimal set to apply the properties of the heat kernel, that is, the fractional Sobolev space $\W^{2-2/q,q}(\O)$. Furthermore, the compatibility condition required for the case $q=3$ is actually slightly weaker than "$\nabla_x v_{in}\cdot n=0$ on $\pa \O$" (see for example \cite{LSU}).
\end{remark}

\subsection{In the literature}

In the last decades, mathematicians dedicated a considerable effort to the question of the existence of solutions for systems of the form \eqref{sku1}--\eqref{sku3}, and particularly for the original system \eqref{sktoriginaltriang}. 

The local (in time) existence of classical solutions was established by Amann in 1990 in the two papers \cite{AmannII}, \cite{AmannIII}. His theorem also provides a criterium to show that these solutions are global: it suffices to prove that the solutions do not blow-up in finite time in suitable Sobolev spaces.

The global existence for the original system \eqref{sktoriginaltriang} has been investigated under various restrictive assumptions. Most results rely on Amann's theory, therefore the problem is to prove bounds in appropriate Sobolev spaces. One of the main difficulties lies in the use of Sobolev embedding theorems in the parabolic estimates, which provide satisfactory results only in low dimension. Therefore, many existing results require strong restrictions on the dimension and/or on the parameters of the system (typically, one assumes that the cross-diffusion is weak, in the sense that the cross-diffusion term $d_\b$ is small compared to some other parameters), see \cite{ChoiLuiYamada03}, \cite{ChoiLuiYamada04}, \cite{LouNiWu}, \cite{MatanoMimura}, \cite{Mimura81}, \cite{Shim}, \cite{Tuoc07}, \cite{Tuoc08}, \cite{Yagi}. See \cite{DesTres} for a more detailed bibliography.

So far, three groups managed to remove this type of assumptions on the dimension and/or parameters, in particular cases: Choi, Lui and Yamada in \cite{ChoiLuiYamada04} (in the presence of self-diffusion in the first equation and in the absence of self-diffusion in the second equation, i.e. $d_\a>0$ and $d_\c=0$) and, very recently, Desvillettes and Trescases in \cite{DesTres} (in the absence of self-diffusion in both equations, i.e. $d_\a=d_\c=0$), and Hoang, Nguyen and Phan in \cite{HNP} (in the presence of self-diffusion in the first equation, i.e. $d_\a>0$). For the original system \eqref{sktoriginaltriang}, our result is the first one treating the case $(d_\a=0,d_\c>0)$. Furthermore, it provides an unifying proof for the cases with and without self-diffusion (in one or both equations).

We now mention some results of existence of global solutions for systems of form \eqref{sku1}--\eqref{sku3}. Wang obtained it in [27] in the presence of self diffusion in the first equation ($d_\a>0$) and in the absence of self diffusion in the second equation ($d_\c=0$), under a condition (depending on the dimension) of smallness of the parameter $d$ w.r.t. the parameter $a$. The case without self-diffusion ($d_\a=d_\c=0$) was solved by Pozio and Tesei in \cite{PozioTesei} under some strong assumption on the reaction coefficients, and by Yamada in \cite{Yamada} under the assumption $a > d$. We also mention the work of Murakawa \cite{Murakawa} in which the reaction terms considered are Lipschitz continuous functions of $u,\,v$ (and no self-diffusion appears). The results in the case without self-diffusion were extended by Desvillettes and Trescases in \cite{DesTres}. More precisely, in \cite{DesTres}, the authors obtain global weak solutions for systems of the form \eqref{sku1}--\eqref{sku3} in the absence of self-diffusion ($d_\a=d_\c=0$) with the following constraint on the parameters: ($\b\ge 1$, $a\le 1$, $d\le 2$) or ($\b\ge 1$, $a<d$). (Note that this indeed includes the original system \eqref{sktoriginaltriang} when $d_\a=d_\c=0$.) The main ingredients of the proof 
 are entropy and duality methods.

In the continuation of \cite{DesTres}, the present paper deals with weak forms of solutions and exploits entropy and duality methods. These methods give rise to $L^p$ estimates for the solution which are quite explicit. Furthermore, considering weak forms of solutions allows us to consider initial data of low regularity (in comparison with most of previous works which rely on Amann's theory). In comparison with \cite{DesTres}, our Theorem includes the cases with self-diffusion terms ($d_\a>0$ and/or $d_\c>0$).  Another improvement is the possibility, for the first time, to consider cross-diffusion terms $v^\b$ which are quite singular functions of $v$ in zero, since we remove the assumption $\b\ge1$ and replace it by the assumption $\b>0$.

Finally, we refer to \cite{DLMT} for systems of the form \eqref{sku1}--\eqref{sku3} presenting in addition a cross-diffusion term in the second equation (non-triangular case).

\subsection{Notations}\label{subsec:notations}
When $T>0$ is fixed, we write $\L^p=\L^p(]0,T[\times\O)$ for any $p\in[1,\infty]$. For $p\in[1,\infty[$, we denote by $\L^{p+}=\L^{p+}(]0,T[\times\O)$ the union of all spaces $\L^{q}(]0,T[\times\O)$ for $q>p$. We write $\H^1_m(\O)$ the space of functions of $\H^1(\O)$ with zero-mean value on $\O$, and we write $\H_m^{-1}(\O)$ its dual space.

We recall that $\D$ is the set of parameters, defined in \eqref{cdt:param}. In the sequel, $C(...)$ always denotes a positive constant, which only depends on its explicitly written arguments and may change from line to line. For example, $C(\O,\D)$ is a positive constant that only depends on $\O$ and $\D$. Furthermore, when it depends on $\D$, the constant $C(...,\D)$ can be chosen continuous in $\D$ on the set $\{\D\in (\R_+^\ast)^{15}\times\R_+\times\R_+^\ast\times\R_+: a\le 1+\alpha, \, d\le 2+\alpha\}$.

\subsection{Plan of the paper}

We will first prove our result of existence under the following condition, which is slightly more restrictive than \eqref{cdt:param},
\begin{equation}\label{cdt:param2}
\text{$\a\ge0$, $d<2+\a$,  $a<1+\a$}.
\end{equation}
More precisely, we will assume condition \eqref{cdt:param2} in Section 
\ref{sec:existence}.\\

In Section \ref{sec:apriori}, we perform formal computations to establish \emph{a priori} estimates 
on the solutions of system \eqref{sku1}--\eqref{sku4}. In Section \ref{sec:approximating}, we define a semi-discrete (in time) scheme designed to be a smooth approximation of system \eqref{sku1}--\eqref{sku4}, and we prove (rigorously) estimates on the solution of this scheme that are independent of the time step. We use these estimates in Section \ref{sec:existence} to pass to the limit in the approximating scheme under condition \eqref{cdt:param2}. This proves the existence of solution and the estimates required in i) when condition \eqref{cdt:param2} is satisfied. Finally, we come to the end of the proof of Theorem \ref{theo_ex_csd} in Section \ref{sec:special}.

\section{\emph{A priori} estimates}\label{sec:apriori}

This section only contains formal computations. They will not be used in the proof of Theorem \ref{theo_ex_csd}, but we believe that these computations are useful to clarify the main tools at the origin of our result of existence. Similar computations will be performed rigorously in Section \ref{sec:approximating} at the level of an approximating scheme.\\

Therefore, suppose that $(u,v)$ is a classical solution of system \eqref{sku1}--\eqref{sku4} for some smooth positive-valued initial data $u_{in}=u_{in}(x)>0$ and $v_{in}=v_{in}(x)>0$, and such that $u$ and $v$ are positive-valued and sufficiently smooth to perform all following computations (for example, in $C^2_c(\R_+\times \overline{\O})$). Fix $T>0$. Assume condition \eqref{cdt:param}. In the following Sections \ref{subsec:basic}--\ref{subsec:apriori_ent}, we compute estimates for the solution $(u,v)$ on $[0,T]\times\O$.

\subsection{Basic estimates}\label{subsec:basic}
\subsubsection{Maximum principle for $v$}
A direct application of the maximum principle to equation \eqref{sku2} gives
\begin{equation}\label{es:Linfty}
\fbox{$\sup_{[0,T]} v(t) \le \max\{ \sup_\Omega v_{in}, \left(\frac{r_v}{r_c}\right)^{1/c} \}.$}
\end{equation}

\subsubsection{Mass conservation for $u$}

We integrate the equation for $u$ on $[0,T]\times\O$:
\begin{equation*}
\int_\O u(T) = \int_\O u_{in} + \int_0^T\int_\O u\, (r_u - r_{a}\, u^a - r_{b}\,v^b).
\end{equation*}
Since (for all $u\ge0$) $r_u u \le C(a,r_u,r_a) + \frac{r_a}{2}u^{1+a}$, we have

\begin{equation}\label{es:L1}
\fbox{$\int_\O u(T) + \frac{r_{a}}{2} \int_0^T\int_\O u^{1+a} \le \int_\O u_{in} + |\O|\,T\, C(a,r_u,r_a).$}
\end{equation}

\subsection{Duality estimate}

We now present an \emph{a priori} estimate obtained thanks to a recent lemma relying on duality methods. This type of lemma was introduced in \cite{PiSc} and has since then showed to be very useful in the context of cross-diffusion (see \cite{BLMP}, \cite{BoudibaPierre}, \cite{DFPV}, \cite{DesLepMou}). We cite here a version coming from \cite{DLMT}.
\begin{lem}
Let $r_u>0$. Let $M:[0,T]\times\Omega\rightarrow\R_+$ be a positive continuous function lower bounded by a positive constant. Smooth nonnegative solutions of the differential inequality 
\begin{align*}
\partial_t u-\Delta (M u)\leq r_u u \text{ on }\Omega,\\
\partial_n (M u)=0,\text{ on }\partial\Omega,
\end{align*}
 satisfy the bound
\begin{align*}
\int_0^T\int_{\O} M u^2\leq\exp(2 r_u T)\times \left( C_\Omega^2\, \|u_{in}\|_{\H^{-1}_m(\Omega)}^2+\langle u_{in}\rangle^2\int_{\O_T} M\right),
\end{align*}
where $\langle u_{in} \rangle$ denote the mean value of $x\mapsto u(0,x)$ on $\Omega$ and $C_\Omega$ is the Poincar\' e-Wirtinger constant.
\end{lem}

We check that $M:=d_u + d_\a u^\a + d_\b v^\b$ is lower bounded by $d_u>0$. We therefore can apply the lemma to equation \eqref{sku1} to get

\begin{equation*}
\int_0^T\int_{\O} [d_u + d_\a u^\a + d_\b v^\b] u^2\leq C(\Omega,T, u_{in},r_u) \times \left( 1+\int_0^T\int_{\O} [d_u + d_\a u^\a + d_\b v^\b]\right).
\end{equation*}
In particular, since all terms in the LHS are nonnegative,
\begin{eqnarray*}
\int_0^T\int_{\O} d_\a u^{\a+2}&\leq C(\Omega,T, u_{in},r_u) \times \left( 1+\int_0^T\int_{\O} [d_u + d_\a u^\a + d_\b v^\b]\right)\\
&\leq C(\Omega,T, u_{in},v_{in},\D)\times \left( 1+\int_0^T\int_{\O} [1 + d_\a u^\a]\right),
\end{eqnarray*}
where we used the $L^\infty$ bound \eqref{es:Linfty} in the last line. As a consequence, using furthermore the inequality (for $\epsilon>0$ small enough, and for all $z\ge0$) $z^\a \le C_\epsilon + \epsilon z^{2+\a}$, we have
\begin{equation}\label{es:dual}
\fbox{$\int_0^T\int_{\O} u^{2+\a}\leq C(\Omega,T, u_{in},v_{in},\D).$}
\end{equation}

\subsection{Entropy estimates}\label{subsec:apriori_ent}

We now present two new estimates which are crucial to obtain weak compactness on the solutions (as they yield a bound for the gradients of the solutions). These estimates are obtained thanks to the introduction of two functionals whose evolution along the flow of the solutions can be controlled. When decreasing, such functionals are called Lyapunov functionals and sometimes "entropies". By a small abuse, we will refer to the resulting estimates as "entropy estimates".

\subsubsection{Entropy estimate for $v$}

For any $p\neq1$, define
$$ E_v(t) := \int_\O \frac{v^p}{p} (t).$$

We compute the derivative

\begin{eqnarray*}
\ud_t E_v (t) = \int_\O \pa_t v \, v^{p-1} (t)\\
= \int_\O v^p (r_v - r_{c}\,v^c - r_d\, u^d) + \int_\O v^{p-1} \Delta_x [(d_v + d_\c v^\c ) \, v],
\end{eqnarray*}
where the last term writes
\begin{eqnarray*}
\int_\O v^{p-1} \Delta_x [(d_v + d_\c v^\c ) \, v] = - \int_\O \nabla_x [v^{p-1}]\cdot \nabla_x [(d_v + d_\c v^\c ) \, v]\\
= - 4 \frac{p-1}{p^2}\int_\O |\nabla_x v^{p/2}|^2 \,[(d_v +d_\c (\c+1)\, v^\c )].
\end{eqnarray*}

We integrate on $t\in[0,T]$:
\begin{eqnarray*}
\int_\O \frac{v^p}{p}(T) \; + \; 4 \frac{p-1}{p^2} \int_0^T \int_\O \left|\nabla_x [v^{p/2}]\right|^2 [(d_v + d_\c(\c+1)\, v^\c )] \\
= \int_\O \frac{v^p}{p}(0) \;+\; \int_0^T\int_\O v^p (r_v - r_{c}\, v^c - r_{d}\,u^d).
\end{eqnarray*}

Now taking $0<p<1$, we get
\begin{equation*}
\int_0^T \int_\O \left|\nabla_x [v^{p/2}]\right|^2
\le C(p,\D) \left( \int_\O v^p(T) \; + \int_0^T\int_\O v^p ( v^c + u^d) \right)
\end{equation*}
and using the $L^\infty$ bound \eqref{es:Linfty} and the dual estimate \eqref{es:dual} with $d\le 2+\alpha$ (given by condition \eqref{cdt:param}),
\begin{eqnarray}\label{es:p_ent}
\int_0^T \int_\O \left|\nabla_x [v^{p/2}]\right|^2
\le C(p,\O,T,u_{in},v_{in},\D) \qquad (\text{for all } 0<p<1).
\end{eqnarray}
Now, let $q\ge 1$. Let us pick some $p\in]0,1[$. We can write $\nabla_x v^{q/2}= (q/p)\,v^{(q-p)/2} \times \nabla_x v^{p/2}\in L^\infty \times L^2$ by \eqref{es:Linfty} and \eqref{es:p_ent}. Therefore
\begin{eqnarray}\label{es:p_ent+}
\fbox{$\int_0^T \int_\O \left|\nabla_x [v^{p/2}]\right|^2
\le C(p,\O,T,u_{in},v_{in},\D) \qquad ($for all $0<p<\infty).$}
\end{eqnarray}

\subsubsection{Entropy estimate for $u$}

Define
$$ E_u(t) := \int_\O \log (1+u) (t).$$

We compute the derivative

\begin{eqnarray*}
\ud_t E_u (t) = \int_\O \frac{\pa_t u}{1+u} (t)\\
= \int_\O \frac{u}{1+u} (r_u - r_{a}\, u^a - r_{b}\,v^b) + \int_\O \frac{1}{1+u} \Delta_x [(d_u + d_\a u^\a + d_\b v^\b ) \, u],
\end{eqnarray*}
where the last term writes
\begin{eqnarray*}
\int_\O \frac{1}{1+u} \Delta_x [(d_u + d_\a u^\a + d_\b v^\b ) \, u] = - \int_\O \nabla_x [\frac{1}{1+u}] \cdot \nabla_x [(d_u + d_\a u^\a + d_\b v^\b ) \, u]\\
= \int_\O \left|\nabla_x [\log(1+u)]\right|^2 (d_u + d_\a(1+\a) u^\a + d_\b v^\b ) + \int_\O d_\b \frac{u}{1+u} \nabla_x [\log(1+u)] \cdot \nabla_x v^\b.
\end{eqnarray*}

We integrate on $t\in[0,T]$:
\begin{equation}\label{comp:ent_u1}\begin{split}
\int_\O \log (1+u) (0) \; + \; \int_0^T \int_\O \left|\nabla_x [\log(1+u)]\right|^2 (d_u + d_\a(1+\a) u^\a + d_\b v^\b ) \\
= \int_\O \log (1+u) (T) \;-\; \int_0^T\int_\O \frac{u}{1+u} (r_u - r_{a}\, u^a - r_{b}\,v^b) \;-\;\int_0^T\int_\O d_\b \frac{u}{1+u} \nabla_x [\log(1+u)] \cdot \nabla_x v^\b .
\end{split}\end{equation}

Using the elementary inequality $2xy\le x^2+y^2$,
\begin{eqnarray*}
\left|\int_0^T\int_\O \frac{u}{1+u} \nabla_x [\log(1+u)] \cdot \nabla_x v^\b \right|
= 2 \left|\int_0^T\int_\O \frac{u}{1+u} \nabla_x [\log(1+u)] \, v^{\b/2} \cdot \nabla_x v^{\b/2} \right| \\
\le \frac{1}{2}\int_0^T \int_\O \left|\nabla_x [\log(1+u)]\right|^2\, v^\b +  2 \int_0^T\int_\O \left(\frac{u}{1+u}\right)^2 \left|\nabla_x v^{\b/2}\right|^2.
\end{eqnarray*}

Reinserting in \eqref{comp:ent_u1} (and using $u/(1+u)\le 1$), we get
\begin{eqnarray*}
\int_\O \log (1+u) (0) \; + \; \int_0^T \int_\O \left|\nabla_x [\log(1+u)]\right|^2 (d_u + d_\a(1+\a) u^\a + \frac{d_\b}{2} v^\b ) \\
\le \int_\O \log (1+u) (T) \;+\; \int_0^T\int_\O (r_u + r_{a}\, u^a + r_{b}\,v^b) \;+\;2 d_\b \int_0^T\int_\O \left|\nabla_x v^{\b/2}\right|^2 .
\end{eqnarray*}

 In the RHS, the first integral and the second term of the second integral are controlled thanks to the $L^1$ estimate \eqref{es:L1} and the last term of the second integral is controlled thanks to the $L^\infty$ bound \eqref{es:Linfty} :
\begin{eqnarray*}
\int_\O \log (1+u) (0) \; + \; \int_0^T \int_\O \left|\nabla_x [\log(1+u)]\right|^2 (d_u + d_\a(1+\a) u^\a + \frac{d_\b}{2} v^\b ) \\
\le C(\O,T,u_{in},a,b,r_u,r_a,r_b) \;+\;2 d_\b \int_0^T\int_\O \left|\nabla_x v^{\b/2}\right|^2 .
\end{eqnarray*}
Finally, the last term is controlled thanks to \eqref{es:p_ent+}, so that
\begin{eqnarray}\label{es:ent_u}
\fbox{$\int_0^T \int_\O \left|\nabla_x [\log(1+u)]\right|^2 (1 + u^\a )
\le C(\O,T,u_{in},\D).$}
\end{eqnarray}

\section{Approximating scheme}\label{sec:approximating}
In this section we define a semi-discrete (in time) scheme intended to approximate system \eqref{sku1}--\eqref{sku4}, and establish rigorously uniform estimates (w.r.t the time step) for the solution of this scheme. Thanks to these uniform estimates, we will be able to pass to the limit in the approximating scheme in the following Section (under condition \eqref{cdt:param2}).

\subsection{Definition of the scheme}
The scheme takes the following form: $(u_0,v_0)$ are given and for $1\leq k\leq N$ ($N=T/\tau$, $T>0$)
\begin{eqnarray}
\label{eq:u_scheme}\frac{u_k-u_{k-1}}{\tau} -\Delta_x [(d_u + d_\a u_k^\a + d_\b v_k^\b)u_k] &=& u_k\, (r_u - r_{a}\, u_k^a - r_{b}\,v_k^b),\text{ on }\Omega,\\
\label{eq:v_scheme}\frac{v_k-v_{k-1}}{\tau} -\Delta_x [(d_v + d_\c v_k^\c)v_k] &=& v_k\, (r_v - r_{c}\,v_k^c - r_d\, u_k^d),\text{ on }\Omega,\\
\label{eq:bc_scheme}\pa_n u_k =\pa_n v_k&=& 0,\text{ on }\pa\Omega.\end{eqnarray}

This scheme was introduced in \cite{DLMT} in a more general setting. More precisely, in \cite{DLMT} systems of the following form are studied :
\begin{equation}\label{eq:scheme_ab}\begin{split}
\frac{U_k-U_{k-1}}{\tau} -\Delta_x [A(U_k)] &= R(U_k),\text{ on }\Omega,\\
\pa_n A(U_k) &= 0,\text{ on }\pa\Omega,
\end{split}\end{equation}
where
\begin{equation}\label{def:A_R}
A:\begin{matrix}
&\R^I_+&\rightarrow& \R_+^I\\
&U =\begin{pmatrix}
u_1 \\ \vdots \\ u_I
\end{pmatrix}
&\mapsto& 
A(U) =
\begin{pmatrix}
a_1(U) u_1 \\ \vdots \\ a_I(U) u_I
\end{pmatrix}
\end{matrix}
\quad \text{and} \quad
R:\begin{matrix}
&\R^I_+&\rightarrow& \R^I\\
&U =\begin{pmatrix}
u_1 \\ \vdots \\ u_I
\end{pmatrix}
&\mapsto& 
R(U)=
\begin{pmatrix}
r_1(U) u_1 \\ \vdots \\ r_I(U) u_I
\end{pmatrix}
\end{matrix}
\end{equation}
satisfy the following assumptions:
\begin{itemize}
\item[\textbf{H1}] For all $i$, the functions $a_i$ and $r_i$ are continuous from $\R_+^I$ to $\R$.
\item[\textbf{H2}] For all $i$, $a_i$ is lower bounded by some positive constant $\underline{a}>0$, and $r_i$ is upper bounded by a positive constant $\overline{r}>0$. That is $a_i(U) \ge \underline{a} > 0$, 
$r_i(U)\le \overline{r}$ for all $U\in \R_+^I$.
\item[\textbf{H3}] $A$ is a homeomorphism from $\R_+^I$ to itself. 
\end{itemize}

Following \cite{DLMT}, we introduce the
\begin{definition}[Strong solution]\label{def:strongsol}
Assume \textbf{H1}, \textbf{H2}, \textbf{H3}. Let $\tau>0$ and let $U_{k-1}$ be a nonnegative vector-valued function in $\L^\infty(\Omega)^I$. A nonnegative vector-valued function $U_k$ is a strong solution of \eqref{eq:scheme_ab} if $U_k$ lies in $\L^\infty(\Omega)^I$, $A(U_k)$ lies in $\H^2(\Omega)^I$ and \eqref{eq:scheme_ab} is satisfied almost everywhere on $\Omega$, resp. $\pa\O$.
\end{definition}

The general theorem from \cite{DLMT} writes
\begin{Theo}\label{th:existence_scheme_it}
Assume \textbf{H1}, \textbf{H2}, \textbf{H3}. Let $\Omega$ be a bounded open set of $\R^m$ with smooth boundary. Fix $T>0$ and an integer $N$ large enough such that $\overline{r}\tau<1/2$, where $\tau:=T/N$ and $\overline{r}$ is the positive number defined in \textnormal{\textbf{H2}}. Fix $\eta>0$ and a vector-valued function $\L^\infty(\Omega)^I\ni U_0\geq \eta$. Then there exists a sequence of positive vector-valued functions $(U_k)_{1\leq k\leq N}$ in $\L^\infty(\Omega)^I$ which solve \eqref{eq:scheme_ab} (in the sense of Definition \ref{def:strongsol}). Furthermore, it satisfies the following estimates :
for all $k\geq 1$ and $p\in[1,\infty[$,
\begin{align}\label{es:depend_on_tau1_ab}
U_k &\in\mathscr{C}^0(\overline{\Omega})^I, \\
U_k &\geq \eta_{A,R,\tau}  \text{ on }\overline{\Omega},\\
\label{es:depend_on_tau3_ab} A(U_k) &\in \W^{2,p}(\Omega)^I,
\end{align}
where $\eta_{A,R,\tau}>0$ is a positive constant depending on the maps $A$ and $R$ and $\tau$, and
\begin{align}
\label{discr_es:dual_ab}
\sum_{k=1}^{N} \tau \int_\Omega \left(\sum_{i=1}^I U_{k,i}\right)\left(\sum_{i=1}^I A_i(U_k)\right) &\leq  C(\Omega,U_0,A, \overline{r}, N\tau),
\end{align}
where $C(\Omega,U_0,A, \overline{r}, N\tau)$ is a positive constant depending only on $\Omega$, $A$, $\overline{r}$, $N\tau$ and $\|U_0\|_{\L^1\cap\H^{-1}_m(\Omega)}$.
\end{Theo}

Let us go back to system \eqref{eq:u_scheme}--\eqref{eq:bc_scheme}. It can be written in the form \eqref{eq:scheme_ab}--\eqref{def:A_R} with $I=2$ and
\begin{equation}\label{def:a_r}\begin{array}{rr}
 a_1(u,v) = d_u + d_\a u^\a + d_\b v^\b,& \qquad  r_1(u,v) = r_u - r_{a}\, u^a - r_{b}\,v^b, \\
a_2(u,v)= d_v + d_\c v^\c, &\qquad r_2(u,v) = r_v - r_{c}\,v^c - r_d\, u^d.
\end{array}\end{equation}

Applying Theorem \ref{th:existence_scheme_it} directly gives rise to the following existence theorem
\begin{cor}\label{cor:existence_scheme_it}
Let $\Omega$ be a bounded open set of $\R^m$ with smooth boundary. Let the parameters of system \eqref{eq:u_scheme}--\eqref{eq:v_scheme} $\D \in (\R_+^\ast)^{15}\times\R_+\times\R_+^\ast\times\R_+$. Fix $T>0$ and an integer $N$ large enough such that $\max(r_u,r_v)\tau<1/2$, where $\tau:=T/N$. Fix $\eta>0$ and a couple of functions $\L^\infty(\Omega)^2\ni (u_0,v_0)\geq \eta$. Then there exists a sequence of couples of positive functions $(u_k,v_k)_{1\leq k\leq N}$ in $\L^\infty(\Omega)^2$ which solve \eqref{eq:u_scheme}--\eqref{eq:bc_scheme} (in the sense of Definition \ref{def:strongsol}). Furthermore, it satisfies the following estimates :
for all $k\geq 1$ and $p\in[1,\infty[$,
\begin{align}\label{es:depend_on_tau1}
(u_k,v_k) &\in\mathscr{C}^0(\overline{\Omega})^2, \\
u_k \geq \eta_{\D,\tau},\; v_k &\geq \eta_{\D,\tau}  \text{ on }\overline{\Omega},\\
\label{es:depend_on_tau3} ([d_u + d_\a u_k^\a + d_\b v_k^\b,]u_k,\,[d_v + d_\c v_k^\c]v_k) &\in \W^{2,p}(\Omega)^2,
\end{align}
where $\eta_{\D,\tau}>0$ is a positive constant depending on $\D$ and $\tau$, and
\begin{align}
\label{discr_es:dual}
\sum_{k=1}^{N} \tau \int_\Omega (1+ u_k^\a)\,u_k^2  &\leq  C(\Omega,u_0,v_0,\D, N\tau),
\end{align}
where $C(\Omega,u_0,v_0,\D, N\tau)$ is a positive constant depending only on $\Omega$, $\D$, $N\tau$ and $\|(u_0,v_0)\|_{\L^1\cap\H^{-1}_m(\Omega)}$.
\end{cor}
\begin{proof}[Proof of Corollary \ref{cor:existence_scheme_it}]
It suffices to check assumptions \textbf{H1}--\textbf{H3} for $A$ and $R$ defined by \eqref{def:a_r}. Assumptions \textbf{H1} and \textbf{H2} are clearly satisfied with $\underline{a}=\min(d_u,d_v)>0$ and $\overline{r}=\max(r_u,r_v)>0$. It remains to prove that the map $A$ is a homeomorphism from $\R_+^2$ to itself. By a monotonicity argument, it is easy to see that the map $A$ is a continuous bijection from $\R_+^2$ to itself. Since the map $A$ is furthermore proper (thanks to the inequality $|A(u,v)|\ge \min(d_u,d_v) |(u,v)|$ for all $u,v\ge0$), it is a homeomorphism.
\end{proof}

\subsection{Uniform estimates}
Let $\mu>0$ and let $U_0=(u_0,v_0)$ be a couple of positive functions in $\L^\infty(\O)^2$ such that $u_0, v_0\ge\mu>0$ a. e. on $\O$. Let $U_k=(u_k,v_k)$ be a solution of \eqref{eq:u_scheme}--\eqref{eq:bc_scheme} (in the sense of Definition \ref{def:strongsol}). We now establish uniform (w. r. t. $N$) estimates for $U_k$ that will allow us to pass in the limit in the approximating scheme (in Section \ref{sec:existence}). Note that some of these estimates (namely, \eqref{discr_es:pcpmax}, \eqref{discr_es:p_ent+}, \eqref{discr_es:energy_u}) can be seen as the "discretized" (in time) version of the \emph{a priori} estimates \eqref{es:Linfty}, \eqref{es:p_ent+}, \eqref{es:ent_u}, while \eqref{discr_es:dual} can be seen as the "discretized" (in time) version of the \emph{a priori} estimate \eqref{es:dual}. To establish rigorously these uniform estimates, we will use all the time the smoothness of $U_k=(u_k,v_k)$  (for any fixed $N$) specified by estimates \eqref{es:depend_on_tau1}--\eqref{es:depend_on_tau3}.

\subsubsection{Maximum principle}
\begin{lem}\label{le_discr:pcpmax} We have
\begin{eqnarray}\label{discr_es:pcpmax}
\max_{k=1..N} \sup_\Omega v_k \le \max\{ \sup_\Omega v_{0}, \left(\frac{r_v}{r_c}\right)^{1/c} \}.
\end{eqnarray}
\end{lem}
\begin{proof}[Proof of Lemma \ref{le_discr:pcpmax}] Recall equation \eqref{eq:v_scheme}. By the maximum principle, for all $1\le k\le N$, we have 
$ \sup_\Omega v_k \le \max\{ \sup_\Omega v_{k-1}, \left(\frac{r_v}{r_c}\right)^{1/c} \},$
so that \eqref{discr_es:pcpmax} follows by iteration on $k=1..N$.
\end{proof} 

\subsubsection{Entropy estimate for $v$}
\begin{lem}\label{le_discr:p_ent+} Assume condition \eqref{cdt:param}. For all $p\in \R_+^\ast$, we have
\begin{eqnarray}\label{discr_es:p_ent+}
\sum_{k=1}^{N} \tau \int_\O \left| \nabla_x v_k^{p/2}\right|^2 \le C(p,\O,N\tau,u_{0},v_{0},\D),\\
\label{discr_es:energy_v_en0} \sup_{0\le k\le N}\int_\O | \log v_k| \,+\,d_v \sum_{k=1}^{N} \tau \int_\O \left| \nabla_x \log v_k\right|^2 \le \int_\O | \log v_0| \,+\, C(\O,N\tau,u_{0},v_{0},\D),
\end{eqnarray}
where the first constant only depends on $p,\O,N\tau,\|u_0\|_{\L^1\cap\H^{-1}_m(\Omega)},\sup_{\O} v_{0},\D$ and the second constant only depends on $\O,N\tau,\|u_0\|_{\L^1\cap\H^{-1}_m(\Omega)},\sup_{\O} v_{0},\D$.
\end{lem}
\begin{proof}[Proof of Lemma \ref{le_discr:p_ent+}]
Define for all $p\in ]0,1[$ and all $z>0$
\begin{equation}
\phi_p(z)=z-\frac{z^p}{p}-1+\frac{1}{p}\ge 0, \qquad \phi_0(z)=z-\log z >0.
\end{equation}
We have (for all $p\in ]0,1[$ and all $z>0$)
\begin{equation}
\phi_p'(z)=1- z^{p-1},\qquad \phi_0'(z)=1-\frac{1}{z}, \qquad \phi_p''(z)=(1-p) z^{p-2}> 0, \qquad \phi_0''(z)=\frac{1}{z^2}> 0.
\end{equation}
For any $p\in[0,1[$, by convexity of $\phi_p$, for all $y> 0$, $z> 0$ we have $\phi_p'(z) (z-y)\ge \phi_p(z)-\phi_p(y)$. Multiplying equation \eqref{eq:v_scheme} by $\phi_p'(v_k)$ and integrating over $\O$, we therefore get
\begin{equation}\label{comp:enerv}
\int_\O [\phi_p(v_k)-\phi_p(v_{k-1})] - \tau \int_\O \phi_p'(v_k)\Delta_x [(d_v + d_\c v_k^\c)v_k] \le \tau \int_\O \phi_p'(v_k) v_k\, (r_v - r_{c}\, v_k^c - r_{d}\,u_k^d).
\end{equation}
Since $(u_k,v_k)$ satisfy the homogeneous Neumann boundary conditions \eqref{eq:bc_scheme}, we can rewrite the second term as
\begin{eqnarray}
- \tau \int_\O \phi_p'(v_k)\Delta_x [(d_v + d_\c v_k^\c)v_k]
= \tau \int_\O \nabla_x [\phi_p'(v_k)] \cdot \nabla_x [(d_v + d_\c v_k^\c)v_k]\\
= \tau \int_\O \phi_p''(v_k)d_v|\nabla_x v_k|^2+\tau \int_\O \phi_p''(v_k) d_\c (\c+1) v_k^\c|\nabla_x v_k|^2.
\end{eqnarray}
The last term being nonnegative, we have
\begin{eqnarray}
- \tau \int_\O \phi_p'(v_k)\Delta_x [(d_v + d_\c v_k^\c)v_k]
\ge \tau \int_\O \phi_p''(v_k) d_v|\nabla_x v_k|^2,
\end{eqnarray}
so that reinserting in \eqref{comp:enerv} and summing for $k=1..N$
\begin{equation}
\int_\O \phi_p(v_{N}) + \sum_{k=1}^{N} \tau \int_\O \phi_p''(v_k) d_v|\nabla_x v_k|^2 \le \int_\O \phi_p(v_{0}) + \sum_{k=1}^{N} \tau \int_\O \phi_p'(v_k) v_k\, (r_v - r_{c}\, v_k^c - r_{d}\,u_k^d).
\end{equation}
Thanks to estimate \eqref{discr_es:dual} and the assumption $d\le2+\a$ (consequence of condition \eqref{cdt:param}), we can control $u_k^d$ in $L^1$, so that, using furthermore the $L^\infty$ bound \eqref{discr_es:pcpmax} and the continuity of $v\mapsto v\phi_p'(v)$ on $\R_+$, we get
\begin{equation}\label{comp:estimvdisc}
\int_\O \phi_p(v_{N}) + \sum_{k=1}^{N} \tau \int_\O \phi_p''(v_k) d_v|\nabla_x v_k|^2 \le \int_\O \phi_p(v_{0}) + C(p,\O,N\tau,u_{0},v_{0},\D).
\end{equation}
Using the definition of $\phi_0$ and $\phi_p$ for $0<p<1$, we have
\begin{eqnarray}
 \int_\O \phi_p''(v_k) d_v|\nabla_x v_k|^2= \int_\O (1-p) v_k^{p-2} d_v|\nabla_x v_k|^2
= \int_\O (1-p) d_v \frac{4}{p^2} |\nabla_x v_k^{p/2}|^2,\\
 \int_\O \phi_0''(v_k) d_v|\nabla_x v_k|^2= \int_\O v_k^{-2} d_v|\nabla_x v_k|^2
= \int_\O d_v |\nabla_x [\log v_k]|^2,
\end{eqnarray}
so that \eqref{comp:estimvdisc} with $0<p<1$ gives \eqref{discr_es:p_ent+} for $p<1$ and, using furthermore the elementary inequality (for all $z>0$) $|\log z|\le z-\log z$, \eqref{comp:estimvdisc} with $p=0$ gives \eqref{discr_es:energy_v_en0}.

It remains to show \eqref{discr_es:p_ent+} for $p\ge 1$. Let $p\ge 1$, and let us fix $\tilde{p}\in]0,1[$. Combining \eqref{discr_es:pcpmax} and \eqref{discr_es:p_ent+} with $p$ replaced by $\tilde{p}$, we have
\begin{eqnarray}
\sum_{k=1}^{N} \tau \int_\O \left| \nabla_x v_k^{p/2}\right|^2 = \sum_{k=1}^{N} \tau \frac{p^2}{\tilde{p}^2} \int_\O v_k^{p-\tilde{p}} \left| \nabla_x v_k^{\tilde{p}/2}\right|^2
 \le C(p,\O,N\tau,u_{0},v_{0},\D).
\end{eqnarray}
\end{proof}

\subsubsection{Entropy estimate for $u$}

\begin{lem}\label{le_discr:energy_u} We have
\begin{eqnarray}
\label{discr_es:energy_u}\sup_{0\le k\le N}\int_\O u_k \,+\,\sum_{k=1}^{N} \tau \int_\O \left| \frac{\nabla_x u_k}{1+u_k}\right|^2 (1+u_k^\a) \le C(\O,N\tau,u_{0},v_{0},\D),\\
\label{discr_es:energy_u_en0} \sup_{0\le k\le N}\int_\O | \log u_k| \,+\,d_u \sum_{k=1}^{N} \tau \int_\O \left| \nabla_x \log u_k\right|^2 \le \int_\O | \log u_0| \,+ \, C(\O,N\tau,u_{0},v_{0},\D),
\end{eqnarray}
where the constant $C(\O,N\tau,u_{0},v_{0},\D)$ only depends on $\O,N\tau,\|u_0\|_{\L^1\cap\H^{-1}_m(\Omega)},\sup_{\O} v_{0},\D$.
\end{lem}
\begin{proof}[Proof of Lemma \ref{le_discr:energy_u}]
Let $\phi(z)=2z-\log(\mu+z)$ for all $z>0$, with $\mu=0$ or $\mu=1$. It is useful to compute
\begin{equation}
\phi(z)=2z-\log(\mu+z), \qquad \phi'(z)=2-\frac{1}{\mu+z}, \qquad \phi''(z)=\frac{1}{(\mu+z)^2}> 0.
\end{equation}
By convexity of $\phi$, for all $y\ge 0$, $z\ge 0$ we have $\phi'(z) (z-y)\ge \phi(z)-\phi(y)$. Multiplying equation \eqref{eq:u_scheme} by $\phi'(u_k)$ and integrating over $\O$, we therefore get
\begin{equation}
\int_\O [\phi(u_k)-\phi(u_{k-1})] - \tau \int_\O \phi'(u_k)\Delta_x [(d_u + d_\a u_k^\a + d_\b v_k^\b)u_k] \le \tau \int_\O \phi'(u_k) u_k\, (r_u - r_{a}\, u_k^a - r_{b}\,v_k^b).
\end{equation}
Thanks to the homogeneous Neumann boundary conditions \eqref{eq:bc_scheme}, we can rewrite the second term as
\begin{eqnarray}
- \tau \int_\O \phi'(u_k)\Delta_x [(d_u + d_\a u_k^\a + d_\b v_k^\b)u_k]
= \tau \int_\O \nabla_x \phi'(u_k) \cdot\nabla_x [(d_u + d_\a u_k^\a + d_\b v_k^\b)u_k]\\
= \tau \int_\O \nabla_x \phi'(u_k) \cdot[(d_u + d_\a (1+\a) u_k^\a + d_\b v_k^\b) \nabla_x u_k] + \tau \int_\O \nabla_x \phi'(u_k) \cdot [ u_k\nabla_x d_\b v_k^\b]\\
= \tau \int_\O \phi''(u_k) (d_u + d_\a (1+\a) u_k^\a + d_\b v_k^\b) |\nabla_x u_k|^2 + \tau \int_\O \phi''(u_k) u_k\nabla_x u_k  \cdot \nabla_x d_\b v_k^\b.
\end{eqnarray}
Therefore
\begin{eqnarray}\label{comp:ent_log}
\int_\O [\phi(u_k)-\phi(u_{k-1})] + \tau \int_\O \phi''(u_k) (d_u + d_\a (1+\a) u_k^\a + d_\b v_k^\b) |\nabla_x u_k|^2\\
 \le - \tau \int_\O \phi''(u_k) u_k\nabla_x u_k  \cdot \nabla_x d_\b v_k^\b + \tau \int_\O \phi'(u_k) u_k\, (r_u - r_{a}\, u_k^a - r_{b}\,v_k^b).
\end{eqnarray}
The first term of the RHS can be rewritten
\begin{eqnarray}
- \tau \int_\O \phi''(u_k) u_k\nabla_x u_k  \cdot \nabla_x d_\b v_k^\b
 &=& - \tau \int_\O \phi''(u_k) u_k\nabla_x u_k \cdot d_\b \,2\, v_k^{\b/2} \nabla_x  v_k^{\b/2} \\
 &\le & \frac{\tau}{2} \int_\O \phi''(u_k)^2 u_k^2 |\nabla_x u_k|^2  d_\b \, v_k^\b  + 2 \tau \int_\O d_\b \, |\nabla_x  v_k^{\b/2}|^2\\
 &\le & \frac{\tau}{2} \int_\O \phi''(u_k) |\nabla_x u_k|^2  d_\b \, v_k^\b  + 2 \tau \int_\O d_\b \, |\nabla_x  v_k^{\b/2}|^2,
\end{eqnarray}
where we used the elementary inequality $2ab\le a^2+b^2$ (for all $a,b\in \R$) and the bound $\phi''(z) \, z^2 = [z/(\mu+z)]^2\le 1$ (for all $z> 0$). Reinserting in \eqref{comp:ent_log}, we get
\begin{eqnarray}
\int_\O [\phi(u_k)-\phi(u_{k-1})] + \tau \int_\O \phi''(u_k) (d_u + d_\a (1+\a) u_k^\a + \frac{d_\b}{2} v_k^\b) |\nabla_x u_k|^2\\
 \le  2 \tau \int_\O d_\b \, |\nabla_x  v_k^{\b/2}|^2  + \tau \int_\O \phi'(u_k) u_k\, (r_u - r_{a}\, u_k^a - r_{b}\,v_k^b).
\end{eqnarray}
Now using that (for all $z> 0$) $\phi'(z) z (r_u - r_{a}\, z^a) = (2-1/(\mu+z)) z (r_u - r_{a}\, z^a) \le C (r_u, r_a, a)$, we have
\begin{eqnarray}
\int_\O [\phi(u_k)-\phi(u_{k-1})] + \tau \int_\O \phi''(u_k) (d_u + d_\a (1+\a) u_k^\a + \frac{d_\b}{2} v_k^\b) |\nabla_x u_k|^2\\
 \le  2 \tau \int_\O d_\b \, |\nabla_x  v_k^{\b/2}|^2  + \tau |\O| C (r_u, r_a, a).
\end{eqnarray}
Summing for $k=1..N$ we get
\begin{eqnarray}
\int_\O \phi(u_{N})+ \sum_{k=1}^{N} \tau \int_\O \phi''(u_k) (d_u + d_\a (1+\a) u_k^\a + \frac{d_\b}{2} v_k^\b) |\nabla_x u_k|^2\\
 \le  \int_\O \phi(u_{0}) + 2 \sum_{k=1}^{N} \tau \int_\O d_\b \, |\nabla_x  v_k^{\b/2}|^2  + N\tau |\O| C (r_u, r_a, a),
\end{eqnarray}
and using furthermore estimate \eqref{discr_es:p_ent+} with $p=\b$ and $\phi''(z)\ge0$ (for all $z>0$),
\begin{eqnarray}
\int_\O \phi(u_{N})+ \sum_{k=1}^{N} \tau \int_\O \phi''(u_k) (d_u + d_\a (1+\a) u_k^\a) |\nabla_x u_k|^2 \le  \int_\O \phi(u_{0}) + C(\O,N\tau,u_{0},v_{0},\D).
\end{eqnarray}
We conclude \eqref{discr_es:energy_u} by taking $\mu=1$ and using the inequality $2z\ge \phi(z)=2z-\log (1+z)\ge z$ (for all $z\ge0$), and we conclude \eqref{discr_es:energy_u_en0} by taking $\mu=0$ and using the inequality $\phi(z)=2z-\log z\ge |\log z|  $ (for all $z>0$).
\end{proof}

\section{Proof of Existence}\label{sec:existence}
We are now ready to pass to the limit in the approximating scheme \eqref{eq:u_scheme}--\eqref{eq:bc_scheme} in order to obtain a solution of system \eqref{sku1}--\eqref{sku4}, at least when the condition \eqref{cdt:param2} is fulfilled.

\begin{proof}[Proof of i) under condition \eqref{cdt:param2}]
Fix $T>0$. Define the sequence $(u_0,v_0)=(u_0^N,v_0^N):=(\min(u_{in},N)+1/N,v_{in}+1/N)$, so that $(u_0,v_0)$ approximates $(u_{in},v_{in})$ in $(\L^{1}(\Omega)\cap \H^{-1}_m(\O))\times\L^\infty(\Omega)$ (when $N\longrightarrow\infty$), and for all $N$, $(u_0,v_0)\in\L^\infty(\Omega)^2$ and $(u_0,v_0)\geq 1/N>0$. For any $N$ large enough ($N> 2T \max(r_u,r_v)$), use Corollary \ref{cor:existence_scheme_it} to define the family of couple of functions $(u_k,v_k)_{1\le k\le N}$ solving \eqref{eq:u_scheme}--\eqref{eq:bc_scheme} with initial value $(u_0^N,v_0^N)$.

Note that \eqref{discr_es:dual}, \eqref{discr_es:pcpmax}, \eqref{discr_es:p_ent+}, \eqref{discr_es:energy_u}, \eqref{discr_es:energy_u_en0}, \eqref{discr_es:energy_v_en0} are therefore valid with $u_{0}$ (resp. $v_0$) replaced by $u_{in}$ (resp. $v_{in}$) in the constant $C$ in the RHS. If $\log u_{in}$ is in $\L^1(\O)$, we furthermore have that $\log u_{0}$ approximates $\log u_{in}$ in $\L^1(\O)$, so that \eqref{discr_es:energy_u_en0} actually yields a uniform bound (w. r. t. $N$). For the same reason, if $\log v_{in}$ is in $\L^1(\O)$, then \eqref{discr_es:energy_v_en0} actually yields a uniform bound (w. r. t. $N$).

\begin{definition}\label{def:pro}
For $h:=(h_k)_{0\leq k \leq N}$ a given family of functions defined on $\Omega$, we denote by $\underline{h}^{N}$ the step in time function defined on $\R \times \Omega$ by
\begin{align*}
\underline{h}^{N}(t,x):= h_0(x)\mathbf{1}_{]-\tau,0]} (t) + \sum_{k=1}^{N} h_k(x) \mathbf{1}_{](k-1)\tau,k\tau]} (t) + h_N(x) \mathbf{1}_{]N\tau,+\infty[} (t).
\end{align*}
Note that by definition, for all $p\in[1,\infty[$, we have
\begin{align*}
\|\underline{h}^{N}\|_{\L^p(]0,T[\times \O)}=\left(\sum_{k=1}^{N} \tau \int_\Omega |h_{k}(x)|^p\dd x\right)^{1/p}. 
\end{align*}
\end{definition}
With this notation, for any family $h:=(h_k)_{0\leq k \leq N}$ of distributions on $\O$, we have
\begin{equation}
\pa_t \underline{h}^N = h_0(x) \, \otimes \, \delta_{-\tau}(t) + \sum_{k=0}^{N-1} \{ h_{k+1}(x)-h_k(x) \} \, \otimes \, \delta_{k\tau}(t) \quad \text{in } \mathscr{D}'(\R\times \O)/
\end{equation}
In particular, we can rewrite equations \eqref{eq:u_scheme}--\eqref{eq:v_scheme} as
\begin{eqnarray}\label{eq:sku1dis}
\pa_t \underline{u}^N &=& u_0(x) \, \otimes \, \delta_{-\tau}(t) + \sum_{k=1}^N  \tau \left(\Delta_x [ (d_u+d_\a u_k^\a+d_\b v_k^\b) u_k ] + (r_u-r_a u^a_k-r_bv_k^b)u_k\right)\, \otimes \, \delta_{k\tau}(t),\\
\label{eq:sku2dis}\pa_t \underline{v}^N &=& v_0(x) \, \otimes \, \delta_{-\tau}(t) + \sum_{k=1}^N  \tau \Big(\Delta_x [ (d_v+d_\c v_k^\c) v_k ] + (r_v-r_c v^c_k-r_d u_k^d)v_k\Big)\, \otimes \, \delta_{k\tau}(t).
\end{eqnarray}
We want to pass to the limit when $N=T/\tau \rightarrow \infty$ in the two equations above. Estimates \eqref{discr_es:dual}, \eqref{discr_es:pcpmax}, \eqref{discr_es:p_ent+}, \eqref{discr_es:energy_u} and \eqref{discr_es:energy_u_en0}, \eqref{discr_es:energy_v_en0} 
can be respectively rewritten as
\begin{eqnarray}
\label{es:unif1}\int_0^T \int_\Omega (\underline{u}^N)^2\,(1+ (\underline{u}^N)^\a)  \leq  C(\Omega,u_{in},v_{in},\D, T),\\
\label{es:unif2}\sup_{[0,T]} \sup_\Omega \underline{v}^N \le \max\{ \sup_\Omega v_{in}, \left(\frac{r_v}{r_c}\right)^{1/c} \},\\
\label{es:unif3}\text{for all } 0<p<\infty, \quad \int_0^T \int_\O \left| \nabla_x (\underline{v}^N)^{p/2}\right|^2 \le C(p,\O,T,u_{in},v_{in},\D),\\
\label{es:unif4}\sup_{[0,T]} \int_\O \underline{u}^N \,+\,\int_0^T \int_\O \left| \frac{\nabla_x \underline{u}^N}{1+\underline{u}^N}\right|^2 (1+(\underline{u}^N)^\a) \le C(\O,T,u_{in},v_{in},\D),\\
\label{es:unif4bisu}\sup_{[0,T]} \int_\O |\log \underline{u}^N| \,+\,d_u \int_0^T \int_\O \left| \nabla_x \log \underline{u}^N\right|^2 \le  \int_\O | \log u_0^N| \,+ \, C(\O,T,u_{in},v_{in},\D),\\
\label{es:unif4bisv}\sup_{[0,T]} \int_\O |\log \underline{v}^N| \,+\,d_v\int_0^T \int_\O \left| \nabla_x \log \underline{v}^N\right|^2 \le \int_\O |\log v^N_0| \,+\, C(\O,T,u_{in},v_{in},\D).
\end{eqnarray}
Note that all bounds give rise to estimates which are uniform with respect to $N$ (with the additional assumption that $\log u_{in}$, resp. $\log v_{in}$, lies in $\L^1(\O)$ for \eqref{es:unif4bisu}, resp. \eqref{es:unif4bisv}). As a consequence of \eqref{es:unif1}--\eqref{es:unif4}, we have uniformly w.r.t. $N$
\begin{equation}\label{es:unif5}
\underline{u}^N \in \L^{2+\a},\qquad \underline{v}^N \in \L^{\infty},\qquad \nabla_x \underline{u}^N \in \L^{1
}, \qquad \nabla_x \underline{v}^N \in \L^{2}.
\end{equation}
The first bound is a direct consequence of \eqref{es:unif1}, while the third bound is obtained by writing $\nabla_x \underline{u}^N=(1+\underline{u}^N) \times (\nabla_x \underline{u}^N/(1+\underline{u}^N))\in \L^{2+\a}\times\L^2$ thanks to estimates \eqref{es:unif1} and \eqref{es:unif4}. The second and last bounds are a direct consequence of \eqref{es:unif2} and \eqref{es:unif3} with $p=2$. Using furthermore condition \eqref{cdt:param2}, \eqref{es:unif5} leads to
\begin{eqnarray}
(r_u-r_a (\underline{u}^N)^a-r_b(\underline{v}^N)^b)\underline{u}^N \in \L^{1+\mu}, \qquad (r_v-r_c (\underline{v}^N)^c-r_d (\underline{u}^N)^d)\underline{v}^N \in \L^{1+\mu},\\
\label{es:unif5.4}(d_u+d_\a (\underline{u}^N)^\a+d_\b (\underline{v}^N)^\b) \underline{u}^N \in \L^{1+1/(1+\a)}, \qquad (d_v+d_\c (\underline{v}^N)^\c) \underline{v}^N \in \L^{\infty},
\end{eqnarray}
where $1+\mu=(2+\a)/\max(1+a,d)>1$.

Rewriting system \eqref{eq:u_scheme}--\eqref{eq:v_scheme} as the following equalities (which hold in $\mathscr{D}'(]0,T[\times\O)$)
\begin{eqnarray}
\frac{1}{\tau} \{\underline{u}^N -S_\tau \underline{u}^N\} = \Delta_x [(d_u+d_\a (\underline{u}^N)^\a+d_\b (\underline{v}^N)^\b) \underline{u}^N] + (r_u-r_a (\underline{u}^N)^a-r_b(\underline{u}^N)^b)\underline{u}^N,\\
\frac{1}{\tau} \{\underline{v}^N -S_\tau \underline{v}^N\} = \Delta_x [(d_v+d_\c (\underline{v}^N)^\c) \underline{v}^N] + (r_v-r_c (\underline{u}^N)^c-r_d (\underline{u}^N)^d)\underline{v}^N,
\end{eqnarray}
where $S_\tau:u(t,x)\mapsto u(t-\tau,x)$, we finally get
\begin{equation}\label{es:unif6}
\frac{1}{\tau} \{\underline{u}^N -S_\tau \underline{u}^N\} \in \L^{1}(]0,T[,\W^{-2,1}(\O)),\qquad \frac{1}{\tau} \{\underline{v}^N -S_\tau \underline{v}^N\} \in \L^{1}(]0,T[,\W^{-2,1}(\O)).
\end{equation}
The uniform (w.r.t. $N$) bounds \eqref{es:unif5}, \eqref{es:unif6} are sufficient to apply a discrete version of Aubin-Lions lemma (see for example \cite{drejun}) to get the strong convergences (up to a subsequence) when $N\longrightarrow\infty$
\begin{equation}\label{cv:u_v}
\underline{u}^N\longrightarrow u\ge0 \text{ in } \L^{1}, \qquad\underline{v}^N\longrightarrow v\ge 0 \text{ in } \L^{1}.
\end{equation}
Let us first check that these strong convergences allow us to pass to the limit in the uniform estimates \eqref{es:unif1}--\eqref{es:unif4bisv} to get estimates \eqref{th_es1}--\eqref{th_es6v}. We can directly pass to the limit in estimate \eqref{es:unif2}, and we use Fatou's lemma to pass to the limit in estimate \eqref{es:unif1} and in the first terms in estimates \eqref{es:unif4}, \eqref{es:unif4bisu} and \eqref{es:unif4bisv}. To pass to the limit in \eqref{es:unif3}, we first notice that $(\underline{v}^N)^{p/2}$ is uniformly bounded in $\L^{2+}$ thanks to estimate \eqref{es:unif2}, so that it converges weakly in $\L^2$. We conclude by using the weak lower semi-continuity of the $\L^2(]0,T[,\W^{1,2}(\O))$ norm.

To compute the remaining limit in \eqref{es:unif4}, it is convenient to notice that $(1+u^\a)\sim (1+u)^{\a}$ (in the sense that there exist $c_\a$, $C_\a>0$ such that for $u\ge 0$, we have $c_\a (1+u)^{\a} \le (1+u^\a) \le C_\a (1+u)^{\a}$), so that \eqref{es:unif4} actually yields a uniform bound in $\L^1$ for $|\nabla_x \underline{u}^N|^2\, (1+\underline{u}^N)^{\a-2}= 4/\a^2\,|\nabla_x (1+\underline{u}^N)^{\a/2}|^2$ if $\a>0$ (resp. for $|\nabla_x \underline{u}^N|^2\, (1+\underline{u}^N)^{\a-2}= |\nabla_x \log(1+\underline{u}^N)|^2$ if $\a=0$). Since $(1+\underline{u}^N)^{\a/2}$ (resp. $\log(1+\underline{u}^N)$) is uniformly bounded in $\L^{2+}$ thanks to estimate \eqref{es:unif1}, it converges weakly in $\L^2$. Using the weak lower semi-continuity of the $\L^2(]0,T[,\W^{1,2}(\O))$ norm, we get the desired bound.

To compute the remaining limit in \eqref{es:unif4bisu}, we first observe that, thanks to Poincar\'e-Wirtinger's inequality,
\begin{equation*}\begin{split}
\int_0^T\int_\O \left|\log \underline{u}^N\right|^2 & = \int_0^T\int_\O \left||\O|^{-1}\int_\O \log \underline{u}^N\,+\,\log \underline{u}^N\,-\,|\O|^{-1}\int_\O \log \underline{u}^N\right|^2\\
&\le 2 \int_0^T \int_\O|\O|^{-2}\left| \int_\O \log \underline{u}^N\right|^2 \,+ \, 2 \int_0^T\int_\O \left|\log \underline{u}^N-|\O|^{-1}\int_\O \log \underline{u}^N\right|^2\\
&\le 2 T |\O|^{-1} \left(\sup_{[0,T]} \int_\O \left|\log \underline{u}^N\right|\right)^2 \,+ \,2 C(\O)^2 \int_0^T\int_\O \left|\nabla_x\log \underline{u}^N\right|^2,
\end{split}\end{equation*}
so that \eqref{es:unif4bisu} yields a uniform bound for $\log \underline{u}^N$ in $\L^2$. Together with the strong convergence \eqref{cv:u_v}, this implies that $\log \underline{u}^N$ converges towards $\log u$ weakly in $\L^1$, which allows us to pass to the limit in the second term of \eqref{es:unif4bisu} thanks to the weak lower semi-continuity of the $\L^2$ norm. The same arguments allow us to pass to the limit in \eqref{es:unif4bisv}.

It remains to check that $(u,v)$ is a solution of \eqref{sku1}--\eqref{sku4} in the sense of Definition \ref{def:weaksol}. Clearly $(u,v)$ lies in $L^{\max(1+a,d)}\times L^{\infty}$ thanks to estimates \eqref{th_es1}, \eqref{th_es2} and condition \eqref{cdt:param2}. Furthermore, using estimates \eqref{th_es1}--\eqref{th_es5}, it is classical to validate the following computations (thanks to regularized approximations of $(1+u)^{\a/2}$ and $v^\b$) when $\a>0$
\begin{eqnarray*}
\nabla_x [ d_\a\, u^{1+\a} ]
= d_\a\left(2+\frac{2}{\a}\right)\stackrel{\in\L^{2}}{\overbrace{\nabla_x (1+u)^{\a/2}}} \, \stackrel{\in\L^{2}}{\overbrace{u^{\alpha}/(1+u)^{\a/2-1}}}\, \in \L^1,\\
\nabla_x [ (d_u +d_\b v^\b) u ] = \frac{2}{\a}\stackrel{\in\L^{\infty}}{\overbrace{(d_u +d_\b v^\b)}} \, \stackrel{\in\L^{2}}{\overbrace{(1+ u)^{1-\a/2} }}\, \stackrel{\in\L^{2}}{\overbrace{\nabla_x (1+u)^{\a/2} \phantom{\big|}}} + d_\b\stackrel{\in\L^{2}}{\overbrace{\nabla_x  v^\b }} \,\stackrel{\in\L^{2}}{\overbrace{u\phantom{\big|}}}\, \in \L^1,
\end{eqnarray*}
and,  (thanks to regularized approximations of $\log(1+u)$ and $v^\b$) when $\a=0$,
\begin{eqnarray*}
\nabla_x [ (d_u+d_\a +d_\b v^\b) u ] = \frac{2}{\a}\stackrel{\in\L^{\infty}}{\overbrace{(d_u+d_\a +d_\b v^\b)}}\, \stackrel{\in\L^{2}}{\overbrace{(1+ u)\phantom{\big|} }}\,\stackrel{\in\L^{2}}{\overbrace{\nabla_x \log(1+u) \phantom{\big|}}} + d_\b\stackrel{\in\L^{2}}{\overbrace{\nabla_x  v^\b }} \,\stackrel{\in\L^{2}}{\overbrace{u\phantom{\big|}}}\, \in \L^1.
\end{eqnarray*}
Therefore, we also have that $\nabla_x \left[(d_u+ d_\a u^\a+d_\b v^\b)\,u\right]$, $\nabla_x \left[(d_v+ d_\c v^\c)\,v\right]$ lie in $\L^1$. Let $\psi_1$, $\psi_2$ be two test functions in $C^2_c([0,T[ \times {\ov{\Omega}})$. Let us first extend $\psi_1$, $\psi_2$ on $\R_+ \times {\ov{\Omega}}$ by zero; then we extend $\psi_1$, $\psi_2$ on $\R \times {\ov{\Omega}}$ in such a way that $\psi_1$, $\psi_2$ lie in $C^2_c(\R \times {\ov{\Omega}})$. Testing equation \eqref{eq:sku1dis} with $\psi_1$ gives
\begin{eqnarray}
 - \int_{-\infty}^{\infty}\int_{\Omega} (\pa_t \psi_1)\, \underline{u}^N - \int_{\Omega} \psi_1(-\tau,\cdot)\, u_{0} \\
= \sum_{k=1}^N \tau \int_{\Omega} \Delta_x \psi_1 (k\tau) \left[(d_u+ d_\a u_k^\a+d_\b v_k^\b)\,u_k\right] +  \sum_{k=1}^N \tau \int_{\Omega} \psi_1(k\tau)\, u_k\, (r_u - r_{a}\, u_k^a - r_{b}\,v_k^b), 
\end{eqnarray}
where all terms are well defined (all integrands are integrable on their respective domains of integration). We rewrite this equation as
\begin{eqnarray}\label{eq:weak_form_N}
 - \int_{-\infty}^{\infty}\int_{\Omega} (\pa_t \psi_1)\, \underline{u}^N - \int_{\Omega} \psi_1(-\tau,\cdot)\, u_{0} \\
= \int_{-\infty}^\infty \int_{\Omega} \Delta_x \tilde{\psi}_1^N \left[(d_u+ d_\a (\underline{u}^N)^\a+d_\b (\underline{v}^N)^\b)\,\underline{u}^N\right] +  \int_{-\infty}^\infty \int_{\Omega} \tilde{\psi}_1^N\, \underline{u}^N\, (r_u - r_{a}\, (\underline{u}^N)^a - r_{b}\,(\underline{v}^N)^b), 
\end{eqnarray}
where
\begin{eqnarray}
\tilde{\psi}_1^N(t,x):= \sum_{k=1}^N \psi_1(k\tau,x)  \mathbf{1}_{](k-1)\tau,k\tau]} (t) \longrightarrow \psi_1 (t,x) \mathbf{1}_{]0,T]} (t) \quad \text{in } \L^\infty(\R\times\Omega).
\end{eqnarray}
Note that we also have
\begin{eqnarray}
\Delta_x \tilde{\psi}_1^N(t,x)= \sum_{k=1}^N \Delta_x \psi_1(k\tau,x)  \mathbf{1}_{](k-1)\tau,k\tau]} (t) \longrightarrow \Delta_x\psi_1 (t,x) \mathbf{1}_{]0,T]} (t) \quad \text{in } \L^\infty(\R\times\Omega).
\end{eqnarray}
Therefore, thanks to the uniform integrability (given by \eqref{es:unif3}--\eqref{es:unif4}) and the strong convergences \eqref{cv:u_v}, we have the convergences when $N\longrightarrow\infty$
\begin{eqnarray}
 - \int_{-\infty}^{\infty}\int_{\Omega} (\pa_t \psi_1)\, \underline{u}^N \longrightarrow  - \int_{0}^{T}\int_{\Omega} (\pa_t \psi_1)\, u,\\
\int_{-\infty}^\infty \int_{\Omega} \tilde{\psi}_1^N\, \underline{u}^N\, (r_u - r_{a}\, (\underline{u}^N)^a - r_{b}\,(\underline{v}^N)^b) 
\longrightarrow
\int_0^T \int_{\Omega} \psi_1\, u\, (r_u - r_{a}\, u^a - r_{b}\,v^b),\\
\int_{-\infty}^\infty \int_{\Omega} \Delta_x \tilde{\psi}_1^N \left[(d_u+ d_\a (\underline{u}^N)^\a+d_\b (\underline{v}^N)^\b)\,\underline{u}^N\right] 
\longrightarrow 
\int_0^T \int_{\Omega} \Delta_x \psi_1 \left[(d_u+ d_\a u^\a+d_\b v^\b)\,u\right],
\end{eqnarray}
where the last term can be rewritten (since $\nabla_x\left[(d_u+ d_\a u^\a+d_\b v^\b)\,u\right]$ lies in $\L^1$)
\begin{eqnarray}
\int_0^T \int_{\Omega} \Delta_x \psi_1 \left[(d_u+ d_\a u^\a+d_\b v^\b)\,u\right]=- \int_0^T \int_{\Omega} \nabla_x \psi_1 \cdot \nabla_x\left[(d_u+ d_\a u^\a+d_\b v^\b)\,u\right].
\end{eqnarray}
It remains to treat the term coming from the initial data in \eqref{eq:weak_form_N}. By the dominated convergence theorem and the definition of $u_0$,
\begin{eqnarray}
- \int_{\Omega} \psi_1(-\tau,\cdot)\, u_{0} \longrightarrow - \int_{\Omega} \psi_1(0,\cdot)\, u_{in}.
\end{eqnarray}
Replacing these four convergences in \eqref{eq:weak_form_N}, we get that \eqref{eq:weak_form} is satisfied.  A straightforward density argument allows us to replace $\psi_1$ in \eqref{eq:weak_form} by any test function in $C^1_c([0,T[ \times {\ov{\Omega}})$.
Performing very similar computations for equation \eqref{eq:sku2dis} with the test function $\psi_2$, we can finally show that $(u,v)$ is a local (in time) weak solution solution of \eqref{sku1}--\eqref{sku4} on $[0,T]$.

Performing iteratively this proof on $[0,2T]$, $[0,3T]$, ..., we can define $(u,v)$ on $\R_+\times \Omega$ in such a way that $(u,v)$ is a (global) weak solution of \eqref{sku1}--\eqref{sku4} and it satisfies estimates \eqref{th_es1}--\eqref{th_es5} for any $T>0$.
\end{proof}

\section{Special cases}\label{sec:special}
This section is devoted to the end of the proof of Theorem \ref{theo_ex_csd}, that is, the proof of i) when condition \eqref{cdt:param2} is not fulfilled and the proof of ii) and iii).

Recall condition \eqref{cdt:param}. "Condition \eqref{cdt:param2} is not fulfilled" exactly means that
\begin{equation}
\text{$\a=0$ and \{($a= 1$, $d\le 2$) or ($a< 1$, $d=2$)\}}.
\end{equation}
In the subsection below, we will prove together i) and ii) under the wider condition
\begin{equation}\label{cdt:param3}
\text{$\a=0$ and ($a\le 1$, $d\le 2$)},
\end{equation}
which is the condition required in ii) in our Theorem. Note that we therefore freely have a second proof of i) in the case $\a=0$ and $a<1$, $d<2$.

The last subsection is devoted to the proof of iii).

\subsection{The case $\a=0$}\label{subsec:alpha0}
The following Lemma is crucial to establish estimate \eqref{th_es7}. It is adapted from a similar result in \cite{DesTres} (one main difference being that the version stated below tackles weaker forms of solutions), which is itself adapted from an original result in \cite{CDF14}.
\begin{lem}\label{jeveuxunnom}
We consider $T>0$, $\Omega$ a bounded regular open set of $\R^m$ ($m\in \N^*$), $q>1$, a function $R$ in $L^q(]0,T[\times\Omega)$, and a function $M$ in $L^\infty(]0,T[\times\Omega)$ satisfying
\begin{equation}
R(t,x) \le K \qquad \text{and} \qquad 0<m_0\le M (t,x) \le m_1 \qquad \text{for } t\in ]0,T[\text{ and } x\in \Omega,
\end{equation}
for some constants $K>0$, $m_0, m_1>0$. Then, one can find $\nu\in]0,1[$ depending only on $\Omega$ and the constants $m_0$, $m_1$, such that for any initial datum $u_{in}$ in $L^2(\Omega)$, any nonnegative \textbf{very weak} solution $u\in L^{2-\nu}(]0,T[\times\Omega)$ of the system
\begin{equation}\label{syst:duality_lemma}\left\{\begin{aligned}
& \partial_{t}u - \Delta_x (Mu) = R \leq K \text{ in }]0,T[\times\Omega,\\
& u(0,x)=u_{in}(x) \text{ in } \Omega,\\
& \nabla_x (Mu)(t,x)\cdot n(x) = 0 \text{ on } ]0,T[\times\partial\Omega,
\end{aligned}\right.\end{equation}
(in the sense that  for all test functions $\psi \in C^2_c([0,T[\times\overline{\O})$ satisfying $\nabla_x \psi \cdot n = 0$ on $[0,T]\times\partial\Omega$, the equality
\begin{equation}\label{duallem:weak_form}
 \int_0^T\int_\O u \pa_{t} \psi + \int_\O u_{in} \psi(0,\cdot) + \int_0^T\int_\O Mu \Delta_x \psi + \int_0^T\int_\O \psi  R =0
\end{equation}
is verified), satisfies
\begin{equation}\label{dual_estimate}
\|u\|_{L^{2+\nu}(]0,T[\times\Omega)} \le C_T\, \left( \|u_{in}\|_{L^2(\Omega)}+K \right),
\end{equation}
where the constant $C_T>0$ depends only on $\Omega$, $T$ and the constants $m_0$, $m_1$.
\end{lem}

\begin{proof}[Proof of Lemma \ref{jeveuxunnom}]
The proof relies on the study of the dual problem. More precisely since here we consider very weak types of solutions for \eqref{syst:duality_lemma}, the analysis can be done rigorously only through (the dual problem of) a "regularized" version of \eqref{syst:duality_lemma} (that is, a family of systems with smooth data and from which we can recover asymptotically the original system \eqref{syst:duality_lemma} in some sense).
The first step of the proof is to define a "regularized" version of \eqref{syst:duality_lemma} and study its dual problem \eqref{syst:reg_dual_pb}. The second step is to ensure that any solution $u$ (of the original problem) considered is the very limit (in some sense specified later) of the solutions of the "regularized" problem (note that this is a result of uniqueness for the original problem). Finally, the third step is to establish an estimate of the type \eqref{dual_estimate} for the solution of the "regularized" system, so that \eqref{dual_estimate} follows after passage to the limit.

\paragraph{First step: regularization and dual problem.}
Let $(\rho^\epsilon)_{0<\epsilon<1}$ be a family of mollifiers on $\R^{m+1}$ such that for all $0<\epsilon<1$,
\begin{equation}
\rho^\epsilon\ge0, \qquad \rho^\epsilon\in C_c^\infty(\R^{m+1}), \qquad \text{supp}\,\rho^\epsilon \subset B_{m+1}(0,\epsilon), \qquad \int_\R\int_{\R^m} \rho^\epsilon =1.
\end{equation}
We extend $M$ by $(m_0+m_1)/2$ on a layer of width 1 around $]0,T[\times\O$, then extend it by 0 on $\R^{m+1}$, that is,
\begin{equation}
M(t,x)= (m_0+m_1)/2 \;\text{ if } \; 0<\text{dist}((t,x),]0,T[\times\O)<1,\qquad M(t,x)=0 \;\text{ if } \; 1<\text{dist}((t,x),]0,T[\times\O),
\end{equation}
 and we define the family $M^\epsilon:=[\rho^\epsilon\ast M]_{|[0,T]\times\overline{\O}}$ (convolution in $\R^{m+1}$ then restriction on $[0,T]\times\overline{\O}$). Therefore
\begin{eqnarray}\label{prop:M_reg}
\text{for all $\epsilon\in ]0,1[$,}&\qquad M^\epsilon \in C^\infty([0,T]\times\overline{\O}), \;\qquad m_0\le M^\epsilon \le m_1, \\
\text{for all $p\in [1,\infty[$,}&\qquad M^\epsilon\longrightarrow M \;\text{in}\; L^p(]0,T[\times\O) \quad \text{when} \quad \epsilon \longrightarrow 0.
\end{eqnarray}
We also define a family $(u_{in}^\epsilon)_\epsilon$ of smooth functions on $\overline{\O}$, which are identically zero on a ($\epsilon$-dependent) neighborhood of $\pa\O$, and which approximates $u_{in}$ in $\L^2(\O)$.

We are now ready to consider the following problem (which can be seen as a "regularized" version of \eqref{syst:duality_lemma})
\begin{equation}\label{syst:duality_lemma_approx}\left\{\begin{aligned}
& \partial_{t}u^\epsilon - \Delta_x (M^\epsilon u^\epsilon) = R \text{ in }]0,T[\times\Omega,\\
& u^\epsilon(0,x)=u_{in}^\epsilon(x) \text{ in } \Omega,\\
& \nabla_x (M^\epsilon u^\epsilon)(t,x)\cdot n(x) = 0 \text{ on } ]0,T[\times\partial\Omega,
\end{aligned}\right.\end{equation}
Note that since $M^\epsilon$ is $C^\infty$, we can always write $\Delta_x [M^\epsilon u^\epsilon] = \Delta_x M^\epsilon\, u^\epsilon + 2 \nabla_x M^\epsilon \cdot \nabla_x u^\epsilon + M^\epsilon \, \Delta_x u^\epsilon$ in the sense of distributions (for $u^\epsilon$ a distribution). By classical results of the Theory of Linear Parabolic Equations (see for example Theorem 9.1, together with the final sentence in paragraph 9, in Chapter IV in \cite{LSU}), there exists a unique $u^\epsilon$ satisfying $\pa_t u^\epsilon$, $\nabla_x^2 u^\epsilon\,\in \L^{q}(]0,T[\times\O)$ which solves the problem \eqref{syst:duality_lemma_approx} in the strong sense.\\
We now introduce the dual problem
\begin{equation}\label{syst:reg_dual_pb}\left\{\begin{aligned}
& \partial_{t}v^\epsilon + M^\epsilon \Delta_x v^\epsilon = f \text{ in }[0,T]\times\overline{\O},\\
& v^\epsilon(T,x)=0 \text{ in } \overline{\O},\\
& \nabla_x  v^\epsilon(t,x)\cdot n(x) = 0 \text{ on } [0,T]\times\partial\Omega,
\end{aligned}\right.\end{equation}
where $f$ is any smooth function on $[0,T]\times\overline{\O}$. Since $M^\epsilon$ and $f$ are smooth, this problem has a unique classical solution $v^\epsilon\in C^\infty([0,T]\times\overline{\O})$ (see for example Theorem 5.3, Chapter IV in \cite{LSU}). We claim that there exists $\nu_1\in]0,1[$ depending only on $\Omega$, $m_0$, $m_1$ such that for all $r\in[2-\nu_1,2+\nu_1]$ and all $\epsilon\in]0,1[$,
\begin{equation}\label{est:dual_reg}
\|\Delta_x v^\epsilon\|_{\L^r(]0,T[\times\O)} \le C(\O,m_0,m_1,p) \|f\|_{\L^r(]0,T[\times\O)}\quad \text{and} \quad \|v^\epsilon(0,\cdot)\|_{\L^2(\O)} \le C(\O,m_0,m_1,p,T) \|f\|_{\L^r(]0,T[\times\O)}.
\end{equation}
Lemma 2.2 (together with Remark 2.3) in \cite{CDF14} states that for any $r\in]1,2[,$ if (with the notations of \cite{CDF14})
\begin{equation}\label{cdtn:almost2}
C_{\frac{m_0+m_1}{2},r}\frac{m_1-m_0}{2}<1,
\end{equation}
then \eqref{est:dual_reg} is true. The proof of Lemma 2.2 (and Remark 2.3) never uses the fact $r<2$, so we can use that \eqref{cdtn:almost2}$\implies$\eqref{est:dual_reg} for any $r>1$. It therefore suffices to check \eqref{cdtn:almost2} for $|r-2|$ small. This is done in the case $r<2$ in the proof of Lemma 3.2 in \cite{CDF14}. For $r>2$, following the ideas of the proof of Lemma 3.2, an appropriate interpolation leads to the bound for all $m>0$, $4>r>2$
\begin{equation}
C_{m,r} \le m^{-\theta} C_{m,4}^{1-\theta},\qquad \text{where} \qquad \frac{\theta}{2}+\frac{1-\theta}{4}=\frac{1}{r},
\end{equation}
that is, for all $m>0$, $4>r>2$,
\begin{equation}
C_{m,p} \le m^{1-4/r} C_{m,4}^{2-4/r}.
\end{equation}
Therefore for all $2<r<4$,
\begin{equation}
C_{\frac{m_0+m_1}{2},r} \frac{m_1-m_0}{2} \le \left(\frac{m_0+m_1}{2}\right)^{1-4/r} C_{\frac{m_0+m_1}{2},4}^{2-4/r} \frac{m_1-m_0}{2}.
\end{equation}
The RHS is a numerical function depending only on $\Omega$, $m_0$, $m_1$ and $r$, and it tends to $\frac{m_1-m_0}{m_0+m_1}<1$ when $r>2$ tends to 2. Therefore there exists a small $\nu_1>0$ depending only on $\Omega$, $m_0$, $m_1$ such that the RHS is $<1$ for all $2<r\le 2+\nu_1$. This implies \eqref{cdtn:almost2} for all $2<r\le 2+\nu_1$.

\paragraph{Second step: uniqueness.}
Let $p\in [2,2+\nu_1[$ and let $u$ be a very weak solution of \eqref{syst:duality_lemma} such that $u\in \L^{p'}$ (where $1/p+1/p'=1$).
 We claim that $u$ is the (unique) strong limit in $\L^{(2+\nu_1)'}$ of the family $(u^\epsilon)_\epsilon$ defined in \eqref{syst:duality_lemma_approx}.

Since $v^\epsilon$ is a smooth function satisfying the homogeneous Neumann boundary condition, we can use it as a test function in \eqref{duallem:weak_form} and against \eqref{syst:duality_lemma_approx}. Subtracting the two equalities thus obtained, we get
\begin{equation}\label{comp:noname}
 \int_0^T\int_\O (u^\epsilon - u) \pa_{t} v^\epsilon + \int_\O (u^\epsilon_{in} - u_{in}) v^\epsilon(0,\cdot) + \int_0^T\int_\O (M^\epsilon u^\epsilon - Mu) \Delta_x v^\epsilon =0.
\end{equation}
Using the smoothness of $v^\epsilon$, $f$ and using $u^\epsilon, u \in \L^{1}$, we can perform rigorously the following computations: we use the definition of $v^\epsilon$ in \eqref{syst:reg_dual_pb}
\begin{equation*}
 \int_0^T\int_\O (u^\epsilon - u) \pa_{t} v^\epsilon = -\int_0^T\int_\O M^\epsilon (u^\epsilon - u) \Delta_x v^\epsilon + \int_0^T\int_\O (u^\epsilon - u) f,
\end{equation*}
and replace in \eqref{comp:noname} to get
\begin{equation*}
\int_0^T\int_\O (u^\epsilon - u) f + \int_\O (u^\epsilon_{in} - u_{in}) v^\epsilon(0,\cdot) + \int_0^T\int_\O (M^\epsilon - M)u \Delta_x v^\epsilon =0.
\end{equation*}
Using \eqref{est:dual_reg} with $r=2+\nu_1$ and H\"older's inequality, it yields
\begin{eqnarray*}
\left|\int_0^T\int_\O (u^\epsilon - u) f\right| &\le& \left|\int_\O (u^\epsilon_{in} - u_{in}) v^\epsilon(0,\cdot) \right|+ \left|\int_0^T\int_\O (M^\epsilon - M)u \Delta_x v^\epsilon\right|\\
&\le& C(\O,m_0,m_1,T) \left( \|u^\epsilon_{in} - u_{in}\|_{\L^2(\O)} + \|(M^\epsilon - M)u\|_{\L^{(2+\nu_1)'}(]0,T[\times\O)} \right) \|f\|_{\L^{2+\nu_1}(]0,T[\times\O)}.
\end{eqnarray*}
By duality, we obtain
\begin{eqnarray*}
\|u^\epsilon - u\|_{\L^{(2+\nu_1)'}(]0,T[\times\O)}
\le C(\O,m_0,m_1,T) \left( \|u^\epsilon_{in} - u_{in}\|_{\L^2(\O)} + \|(M^\epsilon - M)u\|_{\L^{(2+\nu_1)'}(]0,T[\times\O)} \right).
\end{eqnarray*}
Using again H\"older's inequality, we end up with
\begin{eqnarray*}
\|u^\epsilon - u\|_{\L^{(2+\nu_1)'}(]0,T[\times\O)}
\le C(\O,m_0,m_1,T) \left( \|u^\epsilon_{in} - u_{in}\|_{\L^2(\O)} + \|M^\epsilon - M\|_{\L^{s}(]0,T[\times\O)} \|u\|_{\L^{p'}(]0,T[\times\O)} \right),
\end{eqnarray*}
where $1/(2+\nu_1)+1/s=1/p$. Letting $\epsilon$ tend to zero, we get that $\|u^\epsilon - u\|_{\L^{(2+\nu_1)'}(]0,T[\times\O)}\rightarrow 0$.

\paragraph{Third step: uniform bound for the regularized problem.}
Let us choose $f\le0$ in \eqref{syst:reg_dual_pb}. By the maximum principle, we have $v^\epsilon\ge0$. Since $v^\epsilon$ is a smooth function, we can use it as a test function against \eqref{syst:duality_lemma_approx}. We get
\begin{equation}\label{comp:noname1}
 \int_0^T\int_\O u^\epsilon \pa_{t} v^\epsilon + \int_\O u_{in}^\epsilon v^\epsilon(0,\cdot) + \int_0^T\int_\O M^\epsilon u^\epsilon  \Delta_x v^\epsilon + \int_0^T\int_\O R v^\epsilon =0.
\end{equation}
As before, using the smoothness of $v^\epsilon$, $f$ and using $u^\epsilon \in \L^{1}$, we can rigorously compute
\begin{equation}
 \int_0^T\int_\O u^\epsilon \pa_{t} v^\epsilon = -\int_0^T\int_\O M^\epsilon u^\epsilon \Delta_x v^\epsilon + \int_0^T\int_\O u^\epsilon f,
\end{equation}
and replace in \eqref{comp:noname1} to get
\begin{equation}\label{comp:noname2}
-\int_0^T\int_\O u^\epsilon f= \int_\O u^\epsilon_{in} v^\epsilon(0,\cdot) + \int_0^T\int_\O R v^\epsilon \le \int_\O u^\epsilon_{in} v^\epsilon(0,\cdot) + K \int_0^T\int_\O v^\epsilon,
\end{equation}
where we have used the bound $R\le K$ and the nonnegativity of $v^\epsilon$. The last term can be rewritten
\begin{equation}
K \int_0^T\int_\O v^\epsilon = K \int_0^T\int_0^t\int_\O \pa_t v^\epsilon = K \int_0^T\int_0^t\int_\O [-M^\epsilon \Delta_x v^\epsilon +f].
\end{equation}
Replacing in \eqref{comp:noname2} and using \eqref{est:dual_reg} with $r=2-\nu_1$ and H\"older's inequality, it gives
\begin{eqnarray}
\int_0^T\int_\O u^\epsilon (- f) 
\le C(\O,m_0,m_1,T) \left( \|u^\epsilon_{in}\|_{\L^2(\O)} + K \right) \|f\|_{\L^{2-\nu_1}(]0,T[\times\O)}  .
\end{eqnarray}

Using the strong convergence of $u^\epsilon$ in $L^{1}$ and the smoothness of $f$ to pass to the limit in the LHS, we get
\begin{eqnarray}
\int_0^T\int_\O u (- f) 
\le C(\O,m_0,m_1,T) \left( \|u_{in}\|_{\L^2(\O)} + K \right) \|f\|_{\L^{2-\nu_1}(]0,T[\times\O)}.
\end{eqnarray}
Since $u\ge0$ by assumption, this yields by duality
\begin{eqnarray}
\|u\|_{\L^{(2-\nu_1)'}(]0,T[\times\O)}
\le C(\O,m_0,m_1,T) \left( \|u_{in}\|_{\L^2(\O)} + K \right).
\end{eqnarray}

This concludes the proof for some $0<\nu<\min\{(2-\nu_1)'-2,2-(2+\nu_1)'\}$.
\end{proof}

We are now ready to perform the
\begin{proof}[End of proof of i) and proof of ii)]
Let $\a=0$ and $a$, $d$ satisfy condition \eqref{cdt:param3}. Let $(a_\epsilon)_\epsilon\subset [a/2,a]$ and $(d_\epsilon)_\epsilon\subset [d/2,d]$ be two (strictly) increasing families of real numbers such that $a_\epsilon\rightarrow a$ and $d_\epsilon\rightarrow d$ when $\epsilon\rightarrow 0$. This implies that for all $\epsilon>0$, $a_\epsilon<1=1+\a$ and $d_\epsilon<2=2+\a$, so that condition \eqref{cdt:param2} is satisfied with $(a,d)$ replaced by $(a_\epsilon,d_\epsilon)$. Therefore, we can apply the results of Section \ref{sec:existence} with this choice of parameters. For all $\epsilon>0$, there exists a weak solution $(u_\epsilon\ge0,v_\epsilon\ge0)$ of \eqref{sku1}--\eqref{sku4} with $(a,d)$ replaced by $(a_\epsilon,d_\epsilon$). Furthermore, for any fixed $T>0$, this solution satisfies estimates \eqref{th_es1}--\eqref{th_es5} (and \eqref{th_es6u}, resp. \eqref{th_es6v} when $\log u_{in}$, resp. $\log v_{in}$, is in $\L^1(\O)$) with $(u,v)$ replaced by $(u_\epsilon,v_\epsilon)$. Since the constant $C(...,\D)$ is chosen to be continuous in $\D$ on the set $\{\D:a\le1+\a,\,d\le2+\a\}$, the estimates actually give uniform bounds w.r.t. $\epsilon$. As a consequence, we have the following uniform (w.r.t. $\epsilon$) bounds (for all $p>0$)
\begin{equation}\label{es:epsilon1}
u_\epsilon \in \L^{2},\qquad v_\epsilon \in \L^{\infty},\qquad \nabla_x \log (1+u_\epsilon) \in \L^{2}, \qquad \nabla_x v^p_\epsilon \in \L^{2}.
\end{equation}
Let us check that we can apply Lemma \ref{jeveuxunnom} to $u_\epsilon$. We define $R_\epsilon:=(r_u-r_a u_\epsilon^{a_\epsilon}-r_b v_\epsilon^b)u_\epsilon$ and $M_\epsilon:=d_u+d_\a+d_\b v_\epsilon^\b$. For all $\epsilon>0$, $R_\epsilon\in \L^{q_\epsilon}$ where $q_\epsilon:=2/(1+a_\epsilon)>1$, and $$R_\epsilon \le \sup_{u\ge0}(r_u-r_a u^{a_\epsilon})u=r_u\frac{a_\epsilon}{a_\epsilon+1}\left(\frac{r_u}{r_a(a_\epsilon+1)}\right)^{1/a_\epsilon}\le r_u\left(\frac{r_u}{r_a}\right)^{1/a_\epsilon}\le K,$$
where $K:= r_u\left(\frac{r_u}{r_a}\right)^{2/a}$ if $r_u>r_a$, $K:= r_u\left(\frac{r_u}{r_a}\right)^{1/a}$ if $r_u<r_a$. Thanks to the bound \eqref{th_es1} for $v_\epsilon$, we also have for all $\epsilon>0$, $M_\epsilon \in L^\infty$ and
$$ 0<m_0:= d_u+d_\a \le M_\epsilon \le d_u+d_\a+d_\b \max\left\{ \sup_\Omega v^\b_{in}, \left(\frac{r_v}{r_c}\right)^{\b/c} \right\} =:m_1<\infty.$$
We can therefore apply Lemma \ref{jeveuxunnom}, which yields the bound for all $\epsilon>0$
\begin{equation}\label{es:epsilon2}
\|u_\epsilon\|_{L^{2+\nu}} \le C(T,\O,m_0,m_1)\, \left( \|u_{in}\|_{L^2(\Omega)}+K \right),
\end{equation}
where $\nu=\nu(\O,m_0,m_1)>0$. Note that the constants $K$, $m_0$ and $m_1$ being independent of $\epsilon$, this bound is also independent of $\epsilon$.

From \eqref{es:epsilon1}, \eqref{es:epsilon2}, we have the uniform (w. r. t. $\epsilon$) bounds
\begin{eqnarray}
\label{es:unif5.1eps}
\nabla_x u_\epsilon \in \L^{1+\nu/(4+\nu)},& \qquad \nabla_x v^\epsilon \in \L^{2},&\\
\label{es:unif5.4eps}(r_u-r_a u_\epsilon^{a_\epsilon}-r_b v_\epsilon^b)u_\epsilon \in \L^{1+\nu/2},& \qquad (r_v-r_c v_\epsilon^c-r_d u_\epsilon^{d_\epsilon})v_\epsilon \in \L^{1+\nu/2},&\\
\label{es:unif5.5eps}
(d_u+d_\a+d_\b v_\epsilon^\b) u_\epsilon \in \L^{2+\nu},& \qquad (d_v+d_\c v_\epsilon^\c) v_\epsilon \in \L^{\infty},&
\end{eqnarray}
where we used the computation $\nabla_x u_\epsilon= (1+u_\epsilon)\nabla_x \log(1+u_\epsilon)\in \L^{2+\nu}\times\L^2$ for \eqref{es:unif5.1eps} and the assumptions on the parameters $a_\epsilon<a \le 1$ and $d_\epsilon<d \le 2$ for \eqref{es:unif5.4eps}. Using the equations of $(u_\epsilon,v_\epsilon)$, estimates \eqref{es:unif5.4eps}--\eqref{es:unif5.5eps} yield a uniform (w. r. t. $\epsilon$) bound for $\pa_t u_\epsilon$, $\pa_t v_\epsilon$ in $\L^{1+\nu/2}(]0,T[,\W^{-2,1+\nu/2}(\O))$. Combined with the gradient estimates \eqref{es:unif5.1eps}, this allows us to apply Aubin-Lions Lemma, so that, up to a subsequence, when $\epsilon \longrightarrow 0$,
\begin{equation}\label{cv_eps:u_v}
u_\epsilon\longrightarrow u\ge0 \text{ in } \L^{1}, \qquad v_\epsilon\longrightarrow v\ge 0 \text{ in } \L^{1}.
\end{equation}
We obtain estimates \eqref{th_es1}--\eqref{th_es6v} for $(u,v)$ with the same arguments as in the passage to the limit in estimates \eqref{es:unif1}--\eqref{es:unif4bisv} in Section \ref{sec:existence}, and we obtain \eqref{th_es7} thanks to \eqref{es:epsilon2} and Fatou's lemma. As in Section \ref{sec:existence}, estimates \eqref{th_es1}--\eqref{th_es5} enable us to check that $(u,v)\in\L^{\max(1+a,d)}\times\L^\infty$ with $\nabla_x[(d_u+d_\a+d_\b v^\b)u]$ and $\nabla_x[(d_v+d_\c v^\c)v]$ in $\L^1$. Using furthermore estimates \eqref{es:unif5.4eps} and \eqref{es:unif5.5eps}, it is classical to check that $(u,v)$ is a global weak solution of \eqref{sku1}--\eqref{sku4}.
\end{proof}

\subsection{The case $\c=0$}\label{subsec:gamma0}
\begin{proof}[Proof of iii)]
When $\c=0$, the system satisfied (in the weak sense) by $v$ can be rewritten as
\begin{align}
\pa_t v - \, d_v' \Delta_x v  = f \qquad \text{in } \R_+\times\Omega,\\
\nabla_x v\cdot n = 0 \qquad \text{on } \R_+\times\pa\Omega,\\
v(0,\cdot) = v_{in} \qquad \text{in } \Omega,
 \end{align}
where $d_v'=d_v+d_\c>0$, $f=v\, (r_v - r_{c}\,v^c - r_d\, u^d)\in\L^{q}_{loc}(\R_+\times\overline{\O})$ (thanks to the estimates \eqref{th_es1}, \eqref{th_es2} and \eqref{th_es7}) and $v_{in}\in W^{2,q}(\O)$ (by assumption). Using the properties of the heat kernel, we get the $\L^{q}_{loc}(\R_+\times\overline{\O})$ bounds for $\pa_t v$ and $\nabla_x^2 v$ stated in \eqref{th_es8}. The bound for $\nabla_x v$ is obtained by interpolating \eqref{th_es1} and the bound for $\nabla_x^2 v$.
\end{proof}

\medskip

{\bf{Acknowledgement}}: 
The research leading to this paper was funded by the french "ANR blanche" project Kibord: ANR-13-BS01-0004.
\medskip

\bibliographystyle{abbrv}

\bibliography{Tres15}

\begin{thebibliography}{10}

\bibitem{AmannII}
H.~Amann.
\newblock Dynamic theory of quasilinear parabolic equations. {II}.
  {R}eaction-diffusion systems.
\newblock {\em Differential Integral Equations}, 3(1):13--75, 1990.

\bibitem{AmannIII}
H.~Amann.
\newblock Erratum: ``{D}ynamic theory of quasilinear parabolic systems. {III}.\
  {G}lobal existence'' [{M}ath.\ {Z}.\ {\bf 202} (1989), no.\ 2, 219--250;
  {MR}1013086 (90i:35125)].
\newblock {\em Math. Z.}, 205(2):331, 1990.

\bibitem{BLMP}
M.~Bendahmane, T.~Lepoutre, A.~Marrocco, and B.~Perthame.
\newblock Conservative cross diffusions and pattern formation through
  relaxation.
\newblock {\em Journal de Math{\'e}matiques Pures et Appliqu{\'e}es}, 92(6):651
  -- 667, 2009.

\bibitem{BoudibaPierre}
N.~Boudiba and M.~Pierre.
\newblock Global existence for coupled reaction-diffusion systems.
\newblock {\em J. Math. Anal. Appl.}, 250:1--12, 2000.

\bibitem{CDF14}
J.~Canizo, L.~Desvillettes, and K.~Fellner.
\newblock Improved duality estimates and applications to reaction-diffusion
  equations.
\newblock {\em Communications in Partial Differential Equations},
  39(6):1185--1204, 2014.

\bibitem{ChenJungel06}
L.~Chen and A.~J{\"u}ngel.
\newblock Analysis of a parabolic cross-diffusion population model without
  self-diffusion.
\newblock {\em J. Differential Equations}, 224(1):39--59, 2006.

\bibitem{ChoiLuiYamada03}
Y.~S. Choi, R.~Lui, and Y.~Yamada.
\newblock Existence of global solutions for the
  {S}higesada-{K}awasaki-{T}eramoto model with weak cross-diffusion.
\newblock {\em Discrete Contin. Dyn. Syst.}, 9(5):1193--1200, 2003.

\bibitem{ChoiLuiYamada04}
Y.~S. Choi, R.~Lui, and Y.~Yamada.
\newblock Existence of global solutions for the
  {S}higesada-{K}awasaki-{T}eramoto model with strongly coupled
  cross-diffusion.
\newblock {\em Discrete Contin. Dyn. Syst.}, 10(3):719--730, 2004.

\bibitem{DFPV}
L.~Desvillettes, K.~Fellner, M.~Pierre, and J.~Vovelle.
\newblock Global existence for quadratic systems of reaction-diffusion.
\newblock {\em Adv. Nonlinear Stud.}, 7(3):491--511, 2007.

\bibitem{DesLepMou}
L.~Desvillettes, T.~Lepoutre, and A.~Moussa.
\newblock Entropy, duality, and cross diffusion.
\newblock {\em SIAM Journal on Mathematical Analysis}, 46(1):820--853, 2014.

\bibitem{DLMT}
L.~Desvillettes, T.~Lepoutre, A.~Moussa, and A.~Trescases.
\newblock On the entropic structure of reaction-cross diffusion systems.
\newblock Accepted for publication in \emph{Communications in Partial
  Differential Equations}.

\bibitem{DesTres}
L.~Desvillettes and A.~Trescases.
\newblock New results for triangular reaction cross diffusion system.
\newblock {\em arXiv preprint arXiv:1408.5814}, 2014.

\bibitem{drejun}
M.~Dreher and A.~J{\"u}ngel.
\newblock Compact families of piecewise constant functions in ${L}^p (0, {T};
  {B})$.
\newblock {\em Nonlinear Analysis: Theory, Methods \& Applications},
  75(6):3072--3077, 2012.

\bibitem{HNP}
L.~T. Hoang, T.~V. Nguyen, and T.~V. Phan.
\newblock Self-diffusion and cross-diffusion equations: ${W}^{1,p}$-estimates
  and global existence of smooth solutions.
\newblock arXiv: 1311.6828.

\bibitem{IMN}
M.~Iida, M.~Mimura, and H.~Ninomiya.
\newblock Diffusion, cross-diffusion and competitive interaction.
\newblock {\em J. Math. Biol.}, 53(4):617--641, 2006.

\bibitem{LSU}
O.~A. Ladyzhenskaya, V.~A. Solonnikov, and N.~N. Ural'ceva.
\newblock {\em Linear and quasilinear equations of parabolic type}, volume~23
  of {\em Translations of Mathematical Monographs}.
\newblock Amercian Mathematical Society, Providence, 1968.

\bibitem{LouNiWu}
Y.~Lou, W.-M. Ni, and Y.~Wu.
\newblock On the global existence of a cross-diffusion system.
\newblock {\em Discrete Contin. Dynam. Systems}, 4(2):193--203, 1998.

\bibitem{MatanoMimura}
H.~Matano and M.~Mimura.
\newblock Pattern formation in competition-diffusion systems in nonconvex
  domains.
\newblock {\em Publ. Res. Inst. Math. Sci.}, 19(3):1049--1079, 1983.

\bibitem{Mimura81}
M.~Mimura.
\newblock Stationary pattern of some density-dependent diffusion system with
  competitive dynamics.
\newblock {\em Hiroshima Math. J.}, 11(3):621--635, 1981.

\bibitem{Murakawa}
H.~Murakawa.
\newblock A relation between cross-diffusion and reaction-diffusion.
\newblock {\em Discrete Contin. Dyn. Syst.}, Ser. S 5(1):147--158, 2012.

\bibitem{Okubo}
A.~Okubo.
\newblock {\em Diffusion and ecological problems: mathematical models},
  volume~10 of {\em Biomathematics}.
\newblock Springer-Verlag, Berlin-New York, 1980.
\newblock An extended version of the Japanese edition, {{\i}t Ecology and
  diffusion}, Translated by G. N. Parker.

\bibitem{PiSc}
M.~Pierre and D.~Schmitt.
\newblock Blowup in reaction-diffusion systems with dissipation of mass.
\newblock {\em SIAM J. Math. Anal.}, 28(2):259--269, 1997.

\bibitem{PozioTesei}
M.~A. Pozio and A.~Tesei.
\newblock Global existence of solutions for a strongly coupled quasilinear
  parabolic system.
\newblock {\em Nonlinear Anal.}, 14(8):657--689, 1990.

\bibitem{SKT}
N.~Shigesada, K.~Kawasaki, and E.~Teramoto.
\newblock Spatial segregation of interacting species.
\newblock {\em J. Theoret. Biol.}, 79(1):83--99, 1979.

\bibitem{Shim}
S.-A. Shim.
\newblock Uniform boundedness and convergence of solutions to the systems with
  a single nonzero cross-diffusion.
\newblock {\em J. Math. Anal. Appl.}, 279(1):1--21, 2003.

\bibitem{Tuoc07}
P.~V. Tuoc.
\newblock Global existence of solutions to {S}higesada-{K}awasaki-{T}eramoto
  cross-diffusion systems on domains of arbitrary dimensions.
\newblock {\em Proc. Amer. Math. Soc.}, 135(12):3933--3941 (electronic), 2007.

\bibitem{Tuoc08}
P.~V. Tuoc.
\newblock On global existence of solutions to a cross-diffusion system.
\newblock {\em J. Math. Anal. Appl.}, 343(2):826--834, 2008.

\bibitem{Yagi}
A.~Yagi.
\newblock Global solution to some quasilinear parabolic system in population
  dynamics.
\newblock {\em Nonlinear Anal.}, 21(8):603--630, 1993.

\bibitem{Yamada}
Y.~Yamada.
\newblock Global solutions for quasilinear parabolic systems with
  cross-diffusion effects.
\newblock {\em Nonlinear Anal.}, 24(9):1395--1412, 1995.

\end{thebibliography}

\end{document}